\documentclass[11pt]{article}
\usepackage{amsmath,amsthm,amssymb}
\usepackage[hyperfootnotes=false]{hyperref}
\usepackage{color}
\usepackage{cancel}
\usepackage{titlesec}
\usepackage{enumerate}
\usepackage[numbers]{natbib}
\usepackage[normalem]{ulem}
\usepackage{bm}

\usepackage{makecell}

\titleformat{\paragraph}
{\normalfont\normalsize\bfseries}{\theparagraph}{1em}{}
\titlespacing*{\paragraph}
{0pt}{3.25ex plus 1ex minus .2ex}{1.5ex plus .2ex}

\def\titlerunning#1{\gdef\titrun{#1}}
\makeatletter
\def\author#1{\gdef\autrun{\def\and{\unskip, }#1}\gdef\@author{#1}}

\makeatother

\def\MSC#1{{\renewcommand{\thefootnote}{}%
\footnote{\emph{Mathematics Subject Classification (2010):} #1}}}
\def\keywords#1{\par\medskip
\noindent\textbf{Keywords:} #1}

\usepackage{authblk}

\usepackage{graphicx}
\usepackage{longtable}
\usepackage{adjustbox}
\usepackage{multirow}
\usepackage{subfigure}
\usepackage{algorithm}
\usepackage{algorithmic}
\usepackage{changepage}
\usepackage{collcell}
\usepackage{datatool}
\usepackage{lscape}

\usepackage{booktabs}
\usepackage{tabularx}

\newtheorem{theorem}{Theorem}[section]
\newtheorem{prop}[theorem]{Proposition}
\newtheorem{cor}[theorem]{Corollary}
\newtheorem{lemma}[theorem]{Lemma}

\newtheorem{proposition}[theorem]{Proposition}

\theoremstyle{definition}

\numberwithin{equation}{section}

\def\bF{\mathbb F}

\def\cL{\mathcal L}

\def\cA{\mathcal A}
\def\cC{\mathcal C}
\def\cB{\mathcal B}

\def\cE{\mathcal E}

\def\cH{\mathcal H}

\def\cN{\mathcal N}
\def\cK{\mathcal K}
\def\cO{\mathcal O}
\def\cP{\mathcal P}
\def\cX{\mathcal X}
\def\cY{\mathcal Y}

\def\cV{\mathcal V}
\def\cT{\mathcal T}

\def\cS{\mathcal S}

\def\cQ{\mathcal Q}
\def\cZ{\mathcal Z}
\def\PG{{\rm PG}}
\def\PGaL{{\rm P\Gamma L}}

\def\F{{\mathbb{F}}}

\def\PGL{{\rm PGL}}
\def\PSL{{\rm PSL}}
\def\PGO{{\rm PGO}}
\def\GL{{\rm GL}}
\def\SL{{\rm SL}}

\def\PSO{{\rm PSO}}

\def\fqq{\mathbb{F}_{q^2}}

\def\rk{\mathrm{rk}}
\def\diag{\mathrm{diag}}

\def\w{\omega}

\def\i{\boldsymbol{i}}

\begin{document}

%%%%% To ease editing, add:

\baselineskip=16pt

%%%%%%%%%%%%%%%%
%% In the running head, give an abbreviation of the title. 
\titlerunning{}

\title{On the geometry of the Hermitian Veronese curve and its quasi-Hermitian surfaces}
\author[1]{Michel Lavrauw}
\author[2]{Stefano Lia}
\author[3]{Francesco Pavese}
\affil[1]{Faculty of Engineering and Natural Sciences, Sabanci University, Istanbul, e-mail: michel.lavrauw@sabanciuniv.edu}
\affil[2]{School of Mathematics and Statistics, University College Dublin, Dublin e-mail: stefano.lia@ucd.ie}
\affil[3]{Dipartimento di Meccanica, Matematica e Management, Politecnico di Bari, Bari e-mail: francesco.pavese@poliba.it}

%\author{Michel Lavrauw \and Stefano Lia \and Francesco Pavese}
%∗
\date{}

\maketitle

%\address{}

\bigskip

\MSC{Primary 51E12; Secondary 51E20 51A50}

%%%%%%%%

\begin{abstract}
The complete classification of the orbits on subspaces under the action of the projective stabiliser of (classical) algebraic varieties is a challenging task, and few classifications are complete. We focus on a particular action of $\PGL(2,q^2)$ (and $\PSL(2,q^2)$) arising from the Hermitian Veronese curve in $\PG(3, q^2)$, a maximal rational curve embedded on a smooth Hermitian surface with some fascinating properties. The study of its orbits leads to a new construction of quasi-Hermitian surfaces: sets of points with the same combinatorial and geometric properties as a non-degenerate Hermitian surface.

\end{abstract}

\section{Introduction and preliminaries}
In combinatorics and finite geometry, the study of various actions of (subgroups of) the classical groups has often led to new constructions of interesting (rare) geometric objects. It is an essential feature of the interplay between groups and geometry. A well-known example, due to Jacques Tits from 1962, is the action of the Suzuki group on the points of a 3-dimensional projective space, giving rise to an ovoid (a notion introduced by Beniamino Segre): a set of points which has the same combinatorial and geometric properties as (but inequivalent to) an elliptic quadric. Since then, this idea has matured, and the availability of computer algebra systems has greatly contributed to recent developments; there are many authors who have used so-called ``orbit-stitching'' to obtain new constructions of desirable geometries. 
A natural instance of this idea
is the action of (a subgroup of) the projective stabiliser of (classical) algebraic varieties. The complete classification of the orbits on subspaces under such actions is a challenging task, and few classifications are complete. 

The focus of this paper is a particular action of the groups $\PGL(2,q^2)$ and $\PSL(2,q^2)$ arising from a maximal rational curve embedded on a smooth Hermitian surface with some fascinating properties. The study of its orbits on projective subspaces of $\PG(3, q^2)$ leads, via orbit-stitching, to a new construction of quasi-Hermitian surfaces: sets of points with the same combinatorial and geometric properties as a non-degenerate Hermitian surface. This is the three-dimensional version of a more general object called quasi-Hermitian variety of $\PG(r, q^2)$, the $r$-dimensional projective space over the finite field $\bF_{q^2}$. A {\em quasi-Hermitian variety} of $\PG(r, q^2)$ is a set $\cS$ of points which has the same intersection numbers with hyperplanes as a non-degenerate Hermitian variety of $\PG(r, q^2)$. For $r\geq 3$ this implies that $\cS$ has the size of a non-degenerate Hermitian variety of $\PG(r, q^2)$, see \cite{SchillewaertVoorde2022}.
This concept was introduced by De Winter and Schillewaert in \cite{DS} extending the analogous notion studied for quadrics by De Clerck, Hamilton, O'Keefe, and Penttila, in \cite{DHOP}. Since quasi-Hermitian varieties have two intersection numbers with respect to hyperplanes of $\PG(r, q^2)$, they produce strongly regular graphs and two-weight codes, see \cite{CK}. Indeed, by embedding $\PG(r, q^2)$ as a hyperplane $\Pi$ of $\PG(r+1, q^2)$, and by considering as vertices the points of $\PG(r+1, q^2) \setminus \Pi$ and adjacency if the two vertices span a line meeting $\cH$, a strongly regular graph, say $\Gamma(\cH)$, is obtained. Moreover, by \cite[Theorem 3.5]{CRV}, given two quasi-Hermitian varieties $\cH$ and $\cH'$, the graphs $\Gamma(\cH)$ and $\Gamma(\cH')$ are isomorphic if and only if there is a collineation of $\PG(r+1, q^2)$ fixing $\Pi$ and mapping $\cH$ to $\cH'$. 

The main subject of our investigation, is the set \begin{align*}
\cO = \{(s^{q+1},s^qt,st^q,t^{q+1}):(s,t)\in \PG(1,q^2)\} \subseteq \PG(3,q^2).
\end{align*}
This set is related to a number of classical objects and classification problems. 
In classical geometry, the Segre variety, named after Corrado Segre, is the image of the map 
$$(v_1,\dots,v_l) \in \PG(n_1-1,\mathbb{K})\times \cdots\times \PG(n_l-1,\mathbb{K})\longmapsto v_1\otimes \cdots\otimes v_l \in \PG(N,\mathbb{K})$$ 
that is known as the {\it Segre embedding}. Here $N=n_1\cdots n_l-1$, and $\mathbb K$ is any field.
The group $\GL(V_1)\times \cdots\times \GL(V_n)$ acts on the space $\PG(N,\mathbb{K})$ preserving the Segre variety. The study of its orbits is an important hot topic related to tensors, their decomposition and their rank. Most of the vast literature on this subject is over the reals or algebraically closed fields. Much less is known over finite fields, see for example \cite{LavrauwSheekey2015}, \cite{LavrauwSheekey2017}.

The set $\cO$ we consider in this paper can be seen as a ``twisted version of the twisted cubic'', which is a special case of what is known as the $(d, \bm{\sigma})$-Veronese variety, in the following way.
The {\em $(d, \bm{\sigma})$-Veronese variety} $\cV_{\bm{\sigma},d}$ associated with the automorphisms vector $\bm{\sigma}=(\sigma_0,\ldots,\sigma_{d-1})\in Aut(\mathbb{F}_{q^t}/\mathbb{F}_q)^{d}$, is the image of the map $$\nu_{\bm{\sigma},d}:\PG(n-1, q^t)\longmapsto \PG(n^d-1, q^t)
$$ where
$$\nu_{\bm{\sigma},d}(v)=v^{\sigma_0}\otimes v^{\sigma_1}\otimes\cdots\otimes v^{\sigma_{d-1}}.
$$
This variety has been introduced in \cite{durantelongobardipepe}. 

If $\bm{\sigma} = \bm{1} = (id, \dots, id)$, then $\cV_{\bm{1},d}$ is the well known Veronese variety of degree $d$ of $\PG(n, q^t)$. In particular $\cV_{\bm{1}, d} \subset \PG(N, q^t)$, with $N=\binom{n+d-1}{d}-1$, and clearly, it can be seen as the ``diagonal'' of the Segre variety, i.e. the image of the symmetric version of the Segre embedding. The Veronese varieties $\cV_{\bm{1},d}$ of degree $d$ of the projective line are rational normal curves, which are related to a number of deep problems in finite geometry, especially regarding the MDS conjecture. We refer to \cite{BallLavrauw} for these connections and for the state-of-the art on the MDS conjecture. In particular, the curve $\cV_{\bm{1},2}$ is a conic of the projective plane, whereas $\cV_{\bm{1},3}$ is the twisted cubic of $\PG(3, q^t)$. The orbits of the stabilizer of $\cV_{\bm{1},3}$ on points and planes of $\PG(3, q^t)$ are easily obtained, see for example \cite{BH,FPSO3D}, but the classification of orbits on lines turns out to be a much harder problem, and this problem has recently been the studied in several works,\cite{BPS,ceriap,DMP2021,DMP2022orbits,DMP2022,GL}, with various interesting approaches.

In the case when $\bm{\sigma} = \bm{\tilde{\sigma}} = (1, \sigma, \dots, \sigma^{t-1})$, where $\sigma$ is a generator of $Aut(\mathbb{F}_{q^t}/\mathbb{F}_q)$, then $\cV_{\bm{\tilde{\sigma}}, t}$ is properly contained in $\PG(n^t-1, q)$ and it is the image of the Grassmann embedding of a Desarguesian $(t-1)$-spread of $\PG(nt-1,q)$ \cite{LunardonNormalSpreads}. See also \cite{pepe}. Let $H_{n}(q^2)$ be the set of Hermitian matrices over $\F_{q^2}$. Then $H_n(q^2)$ is an $\F_q$-vector space of dimension $n^2$ and the group $\GL(n, q^2)$ acts on it via $A M (A^q)^t$. Here $A \in \GL(n, q^2)$, $M \in H_n(q^2)$ and $A^q$ denotes the matrix $A$ with all the entries raised to the $q$-th power. In particular $\PG(H_n(q^2)) \simeq \PG(n^2-1, q)$ and the points of $\PG(H_n(q^2))$ corresponding to matrices of given rank are permuted in a single orbit by the group isomorphic to $\PGL(n, q^2)$ induced by $\GL(n, q^2)$. Moreover the set of points of $\PG(H_n(q^2))$ arising from the rank one matrices of $H_n(q^2)$ is called the {\em Hermitian Veronesean} of $\PG(n-1, q^2)$ and it is isomorphic to $\cV_{\bm{\tilde{\sigma}}, 2}$, see \cite{CooThMal, CoSic, HT}, 

In this paper we focus on the Hermitian Veronesean of $\PG(1, q^2)$ (or {\em Hermitian Veronese curve}), namely $\cO$. More precisely, the pointset $\cO$ correspond to the set of rank one matrices of $H_2(q^2)$ via 
\begin{align*}
\begin{pmatrix}
s^{q+1} & s^qt\\ st^q &t^{q+1}
\end{pmatrix} \in H_2(q^2) \longmapsto (s^{q+1},s^qt,st^q,t^{q+1}) \in \cO
\end{align*}
and it is projectively equivalent to an elliptic quadric $\cQ^-(3, q)$ contained in a Baer subgeometry of $\PG(3, q^2)$ \cite[Section 4.3]{HT}.

Finally, as already mentioned, the Hermitian Veronese curve $\cO$ coincides with the set of $\mathbb{F}_{q^2}$-rational points of the curve $\cC=\{(1,t,t^q,t^{q+1}):t \in \bar{\mathbb{F}}_{q^2}\} \cup \{(0,0,0,1)\}$. The curve $\cC$ is a smooth rational curve of degree $q+1$ embedded on the smooth Hermitian surface $\cH(3, q^2)$ with equation $X_1^q X_4 + X_1 X_4^q - X_2^{q+1} - X_3^{q+1} = 0$. Up to the projective equivalence, this is the unique non-planar rational curve of degree $q+1$ embedded on $\cH(3, q^2)$, see \cite{Ojiro}. Note that, from the Natural Embedding Theorem, $\fqq$-maximal curves are precisely, up to a birational isomorphism, the curves of degree $q+1$ embedded on a smooth Hermitian variety, see \cite{NET-KT}. Maximal curves with a suitable automorphism group have recently been employed in \cite{KNS} to construct new examples of particular substructures of the Hermitian surface, called {\em hemisystems}. Notably, the hemisystem found in \cite{CP} can be seen as particular case of this construction, obtained by using the curve $\cC$.

In this paper we are interested in the action of the groups $\PGL(2,q^2)$ and $\PSL(2,q^2)$ on projective subspaces of $\PG(3, q^2)$. In particular we will show that a suitable $\PSL(2, q^2)$-point-orbits sticking leads to a new construction of quasi-Hermitian surfaces of $\PG(3, q^2)$, $q$ odd, see Theorem~\ref{quasiHermitianTHM}.

The paper is organized as follows. %In Section~\ref{SectionPreliminary} we fix the notation and collect some preliminary results, while 
Section~\ref{SectionConstruction} is devoted to the construction of a new family of quasi-Hermitian surfaces and to the proof of our main theorem, namely Theorem~\ref{quasiHermitianTHM}. In Section~\ref{SectionNonIso} we prove that our construction is not isomorphic to any previously known quasi-Hermitian surface. Finally, in Section~\ref{sectionpointorbitdistribution}, we investigate the orbits of the group $\PGL(2,q^2)$ on projective subspaces of $\PG(3, q^2)$, along with some combinatorial invariants as for instance the point-orbit distributions of the planes.

\section{Quasi-Hermitian surfaces}\label{SectionConstruction}

Let $\PG(3, q^2)$ be the three-dimensional projective space equipped with homogeneous projective coordinates $(X_1, X_2, X_3, X_4)$. Throughout the paper $q$ denotes an odd prime power, $\xi$ a primitive element of $\F_{q^2}$ and $\Box_{q}$ the set of non-zero squares of $\F_q$. Denote by $s$ a fixed non-square in $\F_q$ and by $\i$ an element in $\F_{q^2}$ such that $\i^2 = s$. Then $\i + \i^q = 0$ and every element of $\F_{q^2}$ can be uniquely written as $X + \i Y$, for some $X, Y \in \F_q$. 

In this section we determine the orbits of a subgroup $K \cong \PSL(2,q^2)$ of $\PGL(4,q^2)$, stabilising the Hermitian Veronese curve, on points of $\PG(3, q^2)$ and provide a new construction of quasi-Hermitian surfaces of $\PG(3,q^2)$, $q$ odd, see Theorem~\ref{quasiHermitianTHM}.

\subsection{The point orbits of \texorpdfstring{$K \simeq\PSL(2,q^2)\leq\PGL(4,q^2)$}{} and the hypersurfaces \texorpdfstring{$\cS_j$}{} and \texorpdfstring{$\cE_k$}{} of \texorpdfstring{$\PG(3,q^2)$}{}}

Let $\cH(3,q^2)$ be the Hermitian surface defined by 
\begin{align*}
& X_1^q X_4 + X_1 X_4^q - X_2^{q+1} - X_3^{q+1} = 0, 
\end{align*}
and let $\cQ^+(3,q^2)$ be the hyperbolic quadric defined by 
\begin{align*}
& X_1 X_4 - X_2 X_3 = 0.
\end{align*}
Under the Segre embedding 
$$
\left((x_1, y_1), (x_2, y_2)\right) \in \PG(1,q^2) \times \PG(1,q^2) \mapsto (x_1 x_2, x_1 y_2, x_2 y_1, y_1 y_2) \in \cQ^+(3,q^2),
$$ 
the couple of points $\left((x_1^q, y_1^q), (x_1, y_1)\right)$, with $x_1, y_1 \in \F_{q^2}$, is mapped to a point of $\cQ^+(3,q^2)$. Varying $x_1, y_1 \in \F_{q^2}$, we obtain the subset of points of $\cQ^+(3,q^2)$
\[
\cO = \{(1, t, t^q, t^{q+1}) \mid t \in \F_{q^2}\} \cup \{(0,0,0,1)\}
\]
of size $q^2+1$. Let 
$$\Sigma = \{ (a, b, b^q, c) \mid a, c \in \F_q, b \in \F_{q^2} \} \subset \PG(3,q^2).
$$ 
Since the semilinear involution $\tau$ of $\PGaL(4,q^2)$ given by $X_1' = X_1^q, X_2' = X_3^q, X_3' = X_2^q, X_4 = X_4^q$ fixes $\Sigma$ pointwise, it follows that $\Sigma$ is a Baer subgeometry of $\PG(3, q^2)$. On the other hand, there is a projectivity of $\PG(3,q^2)$ mapping $\cO$ to $\{(v, u, u^2-s v^2,1) \mid u, v \in \F_q\} \cup \{(0,0,1,0)\}$. Hence $\cO$ is an elliptic quadric $\cQ^-(3,q)$ of $\Sigma$. Moreover it is easily seen, $\cH(3, q^2) \cap \cQ^+(3, q^2) = \cH(3, q^2) \cap \Sigma = \cQ^+(3, q^2) \cap \Sigma = \cO$. 
%given by 
%$$X_1' = \frac{1}{2 \i}(X_2-X_3), X_2' = \frac{1}{2}(X_2+X_3), X_3' = X_4, X_4' = X_1
%$$ 
%sends $\Sigma$ to $\PG(3,q)$ and $\cO$ to the set 
%$$ 
If $A, B \in \GL(2, q^2)$, then the projectivity of $\PG(3,q^2)$ associated with $A \otimes B \in \GL(4,q^2)$ stabilizes $\cQ^+(3,q^2)$, see also \cite[Theorem 4.112]{HT}. Let 
$A = \begin{pmatrix} 
a & b \\
c & d 
\end{pmatrix} \in \SL(2, q^2)$ and consider the projectivity $g$ of $\PG(3,q^2)$ induced by 
$$
A^q \otimes A = \begin{pmatrix}
a^q & b^q \\
c^q & d^q 
\end{pmatrix} \otimes 
\left( \begin{array}{cc}
a & b \\
c & d 
\end{array} \right) = 
\left( \begin{array}{cccc}
a^{q+1} & a^q b & b^q a & b^{q+1} \\
a^q c & a^q d &  b^q c & b^q d \\
c^q a & c^q b & d^q a & d^q b \\
c^{q+1} & c^q d & d^q c & d^{q+1}
\end{array} \right) \in \SL(4, q^2) . 
$$ 
Then $g$ stabilizes $\cO$ and the subgroup $K$ of $\PGL(4,q^2)$ consisting of projectivities induced by $\{A^q \otimes A \mid A \in \SL(2, q^2)\}$ has order $q^2(q^4-1)/2$ and it is isomorphic to $\PSL(2,q^2)$. In particular $K$ acts in its natural representation on the $q^2+1$ points of $\cO$. The action of $K$ on the points of $\Sigma$ and $\cH(3, q^2)$ is well known (see for instance \cite[p. 343]{CoPa}, \cite{CP}) and described below.

\begin{lemma}
The group $K$ acts transitively on points of $\cO$, has two orbits $\Sigma_1$, $\Sigma_2$ of equal size on points of $\Sigma \setminus \cO$, whose representatives are $S_1=(1,0,0,1)$ and $S_2=(1,0,0,\xi^{q+1})$ and has two orbits $\cH_1$, $\cH_2$ on points of $\cH(3, q^2) \setminus \cO$ of size $q^2(q^2+1)(q-1)/2$, $q^2(q^2+1)(q+1)/2$, respectively.
\end{lemma}
Let $\perp$ be the orthogonal polarity of $\PG(3,q^2)$ associated with $\cQ^+(3,q^2)$. The involution $\tau$ fixes $\cQ^+(3, q^2)$, $\cH_1$, $\cH_2$ and by \cite[p. 5]{CP}, 
\begin{align}
\tau \; \circ \perp \, = \, \perp \circ \; \tau \label{unitary}
\end{align}
is the unitary polarity of $\PG(3, q^2)$ associated with $\cH(3, q^2)$, see also \cite[Section 2.3]{CoPa1}. 

In what follows we investigate the action of the group $K$ on points of $\PG(3, q^2) \setminus (\cH(3, q^2) \cup \Sigma)$. The construction of the quasi-Hermitian surfaces (Theorem~\ref{quasiHermitianTHM}) relies on the observation that the group $K$ actually fixes a family of surfaces, as reported in Lemma~\ref{surfacesstabilized}. A proper choice of elements of this family exhibits a very particular set of geometrical properties, illustrated in the rest of the subsection, leading to the construction of the quasi-Hermitian surfaces as the union of two members of this family. In what follows we collect some useful results regarding the action of the group $K$.

\begin{lemma}\label{surfacesstabilized}
The group $K$ stabilizes $\cH(3, q^2)$, $\cQ^+(3, q^2)$ and the hypersurfaces given by
\[
F_\gamma(X_1,X_2,X_3,X_4) = X_1^q X_4 + X_1 X_4^q - X_2^{q+1} - X_3^{q+1} - \gamma \left(X_1 X_4 - X_2 X_3\right)^{\frac{q+1}{2}} = 0, \quad \gamma \in \F_{q^2}.  
\]
\end{lemma}
\begin{proof}
Let 
\[
H = \begin{pmatrix} 0 & 0 & 0 & 1 \\ 0 & -1 & 0 & 0 \\ 0 & 0 & -1 & 0 \\ 1 & 0 & 0 & 0 \end{pmatrix} \mbox{ and } S = \begin{pmatrix} 0 & 0 & 0 & 1 \\ 0 & 0 & -1 & 0 \\ 0 & -1 & 0 & 0 \\ 1 & 0 & 0 & 0 \end{pmatrix}.
\]
be the Hermitian and symmetric matrix defining $\cH(3, q^2)$ and $\cQ^+(3, q^2)$. If $A \in \SL(2, q^2)$ and $M = A^q \otimes A$, then 
\begin{align}
& M^t H M^q = \det(A)^{q+1} H, \quad M^t S M = \det(A)^{q+1} S. \label{mat}
\end{align}
Hence $K$ leaves invariant both $\cH(3, q^2)$, $\cQ^+(3, q^2)$. Let $g$ be the projectivity of $K$ induced by $M$. If $P = (x_1, x_2, x_2, x_4)$ is a point of $\PG(3, q^2)$ such that $F_{\gamma}(P) = 0$, then 
\[
P H (P^q)^t - \gamma \left( P S P^t \right)^\frac{q+1}{2} = 0.
\]
Hence, by \eqref{mat},
\[
P M^t H M^q (P^q)^t - \gamma \left( P M^t S M P^t \right)^\frac{q+1}{2} = 0.
\]
As $(\det(A)^{q+1})^{\frac{q+1}{2}} = \det(A)^{q+1}$, which holds since $\det(A)^{q+1} \in \Box_{q} \iff \det(A) \in \Box_{q^2}$, it follows that $F_{\gamma}(P^g) = 0$. 
\end{proof}

Consider the following hypersurfaces
\begin{align*}
& \cS_j: & X_1^q X_4 + X_1 X_4^q - X_2^{q+1} - X_3^{q+1} - \left( \xi^{j\frac{q-1}{2}} + \xi^{-j\frac{q-1}{2}} \right) \left(X_1 X_4 - X_2 X_3\right)^{\frac{q+1}{2}} = 0, \\ 
& & j \in \{1, \dots, q\} \setminus \{(q+1)/2\}. \\
& \cE_k: & X_1^q X_4 + X_1 X_4^q - X_2^{q+1} - X_3^{q+1} - \left( \xi^{k\frac{q+1}{2}} + \xi^{-k\frac{q+1}{2}} \right) \left(X_1 X_4 - X_2 X_3\right)^{\frac{q+1}{2}} = 0, \\
& & k \in \{0, \dots, q-1\} \setminus \{(q-1)/2\}. 
\end{align*}

Lemmas \ref{lemma:tec}, \ref{lemma:stab} and \ref{lemma:size} show that these surfaces are indeed pairwise distinct surfaces, each of them contains $\cO$ and is partitioned into two $K$-orbits, namely the set $\cO$ and its complement. 

\begin{lemma}\label{lemma:tec}
Let 
\begin{align*}
& \cZ_1 = \left\{\xi^{j\frac{q-1}{2}} + \xi^{-j\frac{q-1}{2}} \mid j \in \{1, \dots, q\} \setminus \{(q+1)/2\}\right\}, \\ 
& \cZ_2 = \left\{\xi^{k\frac{q+1}{2}} + \xi^{-k\frac{q+1}{2}} \mid k \in \{0, \dots, q-1\} \setminus \{(q-1)/2\}\right\}. 
\end{align*}
Then $|\cZ_1 \cup \cZ_2| = 2(q-1)$.
\end{lemma}
\begin{proof}
Let $z, z' \in \F_{q^2} \setminus \{0\}$ such that $z+\frac{1}{z} = z' + \frac{1}{z'}$, then either $z' = z$ or $z' = \frac{1}{z}$. If $(z, z') = \left(\xi^{j_1\frac{q-1}{2}}, \xi^{j_2\frac{q-1}{2}}\right)$, for some $1 \le j_1 < j_2 \le q$, then either $\xi^{(j_2-j_1)\frac{q-1}{2}} = 1$ or $\xi^{(j_1+j_2)\frac{q-1}{2}} = 1$, where $0 < j_2 \pm j_1 < 2(q-1)$, a contradiction. If $(z, z') = \left(\xi^{k_1\frac{q+1}{2}}, \xi^{k_2\frac{q+1}{2}}\right)$, for some $0 \le k_1 < k_2 \le q-1$, then either $\xi^{(k_2-k_1)\frac{q+1}{2}} = 1$ or $\xi^{(k_1+k_2)\frac{q-1}{2}} = 1$, where $0 < k_2 \pm k_1 < 2(q+1)$, a contradiction. If $(z, z') = \left(\xi^{j\frac{q-1}{2}}, \xi^{k\frac{q+1}{2}}\right)$, for some $1 \le j \le q$, $j \ne (q+1)/2$, $0 \le k \le q-1$, $k \ne (q-1)/2$, then either $\xi^{j\frac{q-1}{2} - k\frac{q+1}{2}} = 1$ or $\xi^{j\frac{q-1}{2} + k\frac{q+1}{2}} = 1$, where $-q(q-1)/2 \le j\frac{q-1}{2} - k\frac{q+1}{2} \le q(q-1)/2$, $j\frac{q-1}{2} - k\frac{q+1}{2} \ne 0$, and $0 < j\frac{q-1}{2} + k\frac{q+1}{2} < q^2-1$, a contradiction.
\end{proof}

\begin{lemma}\label{lemma:stab}
Set 
\begin{align*}
& U = (0, 1, 0, 0), T_1=(\xi, 1, 1, 0), T_2=(\xi^{q-1}, \xi^{q-1}, 1, 0), \\
& R_j = (\xi^j, 0, 0, 1), j \in \{1,\dots, q\}, Q_k = (0, \xi^k, 1, 0), k\in\{1,\dots, q-2\}.
\end{align*}
The following hold:
\begin{align*}
& |K_U| = (q^2-1)/2, \quad |K_{R_j}| = q+1, \quad |K_{Q_k}| = q-1, \quad |K_{T_1}| = |K_{T_2}| = q.
\end{align*}
\end{lemma}
\begin{proof}
It is straightforward to compute $K_U$. 

As for $K_{T_1}$, it is easily seen that an element in $K_{T_1}$ has to fix $(1,0,0,0)$, and we can assume $a=1$, $d\neq0$. The projectivity induced by the matrix
$A = \begin{pmatrix} 
a & b \\
c & d 
\end{pmatrix} \in \SL(2, q^2)$ fixes $T_1$ if and only if, up to a square scalar multiple, $c=0, a=1, d\in\mathbb{F}_{q^2} \setminus \{0\}$ and $b\in\mathbb{F}_{q^2}$. 
\begin{align*}
\begin{cases}
\xi+b+b^q=\xi d\\
d=d^q
\end{cases}
\end{align*}
If $d\neq 1$, then $\xi=(b^q+b)/(d-1)$ and $d = d^q$, which is not possible since $\xi \notin \mathbb{F}_q$. Therefore $d=1$ and $b$ has to be such that $b^q+b=0$. Hence $|K_{T_1}| = q$. 

Consider now $K_{T_2}$. The projectivity associated with the matrix
$\begin{pmatrix} 
1 & 0 \\
0 & \xi 
\end{pmatrix} \in \GL(2, q^2) \setminus \SL(2, q^2)$ maps $T_1$ to $T_2$. Hence $K_{T_2}$ and $K_{T_1}$ are conjugated in $\PGL(2, q^2)$ and $|K_{T_1}| = |K_{T_2}|$.

The size of $K_{R_j}$ is given by the number of solutions up to a scalar multiple in $\mathbb{F}_{q^2}$ of the following system.
\begin{align*}
\begin{cases}
\xi^ja^{q+1}+b^{q+1}=\xi^j(\xi^jc^{q+1}+d^{q+1})\\
\xi^ja^qc+db^q=0\\
\xi^jac^q+bd^q=0;
\end{cases}
\end{align*}
under the condition $ad-bc\in \Box_{q^2}$.
First observe that $a\neq0$, as otherwise, the contradiction $d=0$ and $\xi^{2j}=(b/c)^{q+1}$ would follow. Therefore, we can assume $a=1$ and the system to be 
\begin{align*}
\begin{cases}
\xi^j+b^{q+1}=\xi^j(\xi^jc^{q+1}+d^{q+1})\\
\xi^jc+db^q=0\\
\xi^jc^q+bd^q=0;
\end{cases}
\end{align*}
If $c\neq 0$, then $\xi^j=\frac{-db^q}{c}=\frac{-bd^q}{c^q}$, namely $\xi^{jq}=\xi^j$, which is not possible. Therefore, $c=0$, $b=0$ and $d^{q+1}=1$ follows. This provides $q+1$ admissible solutions. 

Similarly, the size of $K_{Q_k}$ is given by the number of solutions up to a scalar multiple in $\mathbb{F}_{q^2}$ of the system
\begin{align*}
\begin{cases}
\xi^ka^{q}d+b^{q}c=\xi^k(\xi^kbc^{q}+ad^{q})\\
\xi^ka^qb+ab^q=0\\
\xi^kc^qd+cd^q=0;
\end{cases}
\end{align*}
under the condition $ad-bc\in \Box_{q^2}$. %Since $\xi^{2k}$ can not be of the form $ cb^q/(bc^q)$, $a\neq0$. Analogously, since $\xi^{k}$ can not be of the form $ -cd^q/(dc^q)$ or $ -ab^q/(ba^q)$, $c=0$ and $b=0$. In particular, the stabilizer $G_{Q_k}$ is contained in $G_{P_0}\cap G_{P_\infty}$, where $P_0=(1, 0, 0, 0)$ and $P_\infty=(0, 0, 0, 1)$. Hence, 
Some computations show that we can assume $a=1$ and that $b = c = 0$, $\xi^kd=\xi^kd^q$. Hence $|K_{Q_k}| = q-1$.
%we are looking for the number of solutions in $\mathbb{F}_{q^2}$ of the following.
%\begin{equation}
%\begin{cases}
%\\
%\end{cases}
%\end{equation}
%%If $b\neq 0$, it follows $\xi^k=\frac{-b^q}{b}$, namely $\xi^{k(q+1)}=1$, which is not possible. Therefore $b=0$, and
%Namely, there are as many solutions as many values $d\in\mathbb{F}_{q}^*$.
\end{proof}

The main geometrical properties of the surfaces $\cS_j$ and $\cE_k$ are collected in the following Lemmas.

\begin{lemma}\label{lemma:size}
The following hold:
\begin{enumerate}
\item[i)] If $j_1, j_2 \in \{1, \dots, q\} \setminus \{(q+1)/2\}$, $j_1 \ne j_2$, and $k_1, k_2 \in \{0, \dots, q-1\} \setminus \{(q-1)/2\}$, $k_1 \ne k_2$, then  $\cS_{j_1} \cap \cS_{j_2} = \cE_{k_1} \cap \cE_{k_2} = \cS_{j_1} \cap \cE_{k_1} = \cO$.
\item[ii)] $|\cS_j| = \frac{q^2(q^2+1)(q-1)}{2} + q^2+1, \quad j \in \{1, \dots, q\} \setminus \{(q+1)/2\}$.
\item[iii)] $|\cE_k| = \frac{q^2(q^2+1)(q+1)}{2} + q^2+1, \quad k \in \{1, \dots, q-2\} \setminus \{(q-1)/2\}$.
\item[iv)] $|\cE_0| = |\cE_{q-1}| = \frac{q^3(q^2+1)}{2} + q^2+1$.
\end{enumerate}
\end{lemma}
\begin{proof}
Let $\gamma \in \F_{q^2} \setminus \{0\}$ and let $\cC_{\gamma}$ be given by $F_\gamma(X_1,X_2,X_3,X_4) = 0$. If $P$ is a point of $\cC_{\gamma} \cap \cC_{\gamma'}$, then 
\begin{align*}
& 0 = \gamma' F_\gamma(P) - \gamma F_{\gamma'}(P) = (\gamma' - \gamma)(P H (P^q)^t), \\
& 0 = F_\gamma(P) - F_{\gamma'}(P) = (\gamma' - \gamma)(P S P^t)^{\frac{q+1}{2}}.
\end{align*}
Therefore either $\gamma = \gamma'$ and $\cC_\gamma = \cC_{\gamma'}$ or $\gamma \ne \gamma'$ and $\cC_\gamma \cap \cC_{\gamma'} = \cH(3, q^2) \cap \cQ^+(3, q^2) = \cO$. Lemma~\ref{lemma:tec} shows {\em i)}.

Since
\[
\frac{\xi^j + \xi^{jq}}{\xi^{\frac{j(q+1)}{2}}} = \xi^{j\frac{q-1}{2}} + \frac{1}{\xi^{j\frac{q-1}{2}}}, 
\]
the point $R_j = (\xi^j,0,0,1)$ belongs to $\cS_j$ and $|K_{R_{j}}| = q+1$, $j \in \{1, \dots, q\} \setminus \{(q+1)/2\}$; if $k \in \{1, \dots, q-2\} \setminus \{(q-1)/2\}$, the point $Q_k = (0, \xi^k, 1, 0)$ is in $\cE_k$ or in $\cE_{q-1-k}$, according as $q \equiv 1$ or $-1 \pmod{4}$, respectively, and $|K_{Q_k}| = q-1$. Similarly, $T_1 = (\xi, 1, 1, 0)$ belongs to $\cE_{0}$ or $\cE_{q-1}$ if $q \equiv 1$ or $-1 \pmod{4}$ and $T_2 = (\xi^{q-1}, \xi^{q-1}, 1, 0)$ belongs to $\cE_{q-1}$ or $\cE_{0}$ if $q \equiv 1$ or $-1 \pmod{4}$, respectively, where $|K_{T_1}| = |K_{T_2}| = q$. Moreover, since $S_1 \in \cE_0$ and $S_2 \in \cE_{q-1}$, it follows that $\Sigma_ 1 \subset \cE_0$ and $\Sigma_2 \subset \cE_{q-1}$. Since $R_j, Q_k, T_1, T_2 \notin \cO$, by the Orbit-Stabilizer Theorem, it follows that 
\begin{align*}
& |\cS_j| \ge \frac{q^2(q^2+1)(q-1)}{2} + q^2+1, \\
& |\cE_k| \ge \frac{q^2(q^2+1)(q+1)}{2} + q^2+1, \\ 
& |\cE_0| = |\cE_{q-1}| \ge \frac{q(q^4-1)}{2} + \frac{q(q^2+1)}{2} + q^2+1.
\end{align*} 
On the other hand, by {\em i)}, 
\begin{align*}
\sum_{\substack{j=1, \\ j \ne \frac{q+1}{2}}}^{q} |\cS_j \setminus \cO| + \sum_{\substack{k = 0, \\ k \ne \frac{q-1}{2}}}^{q-1} |\cE_k \setminus \cO| \ge q^2(q^2+1)(q^2-q-1) = |\PG(3, q^2) \setminus (\cH(3, q^2) \cup \cQ^+(3, q^2))|.
\end{align*}
The statements follow.
\end{proof}

As a consequence, we can state the following result.

\begin{prop}\label{point-orbits}
The group $K$ has $2q+4$ orbits on points of $\PG(3, q^2)$:
\begin{itemize}
\item[i)] $\cO$;
\item[ii)] $\Sigma_1$ with representative the point $S_1$; 
\item[iii)] $\Sigma_2$ with representative the point $S_2$;
\item[iv)] $\cQ^+(3, q^2) \setminus \cO$ having as a representative the point $U$;
\item[v)] two orbits $\cH_1$, $\cH_2$ on points of $\cH(3, q^2) \setminus \cO$, with representatives $R_{\frac{q+1}{2}}$, $Q_{\frac{q-1}{2}}$ and of size $q^2(q^2+1)(q-1)/2$, $q^2(q^2+1)(q+1)/2$, respectively; 
\item[vi)] $\cS_j \setminus \cO$, with representative the point $R_j$, $j \in \{1, \dots, q\} \setminus \{(q+1)/2\}$;
\item[vii)] $\cE_k \setminus \cO$ if $q \equiv 1 \pmod{4}$ or $\cE_{q-1-k}$ if $q \equiv -1 \pmod{4}$, with representative the point $Q_k$, $k \in \{1, \dots, q-2\} \setminus \{(q-1)/2\}$;
\item[viii)] $\cE_0 \setminus \Sigma$, with representative the point $T_1$ if $q \equiv 1 \pmod{4}$ or $T_2$ if $q \equiv -1 \pmod{4}$;
\item[ix)] $\cE_{q-1} \setminus \Sigma$, with representative the point $T_1$ if $q \equiv -1 \pmod{4}$ or $T_2$ if $q \equiv 1 \pmod{4}$.
\end{itemize}
\end{prop} 

%Note that $\cH_1 = \cS_{(q+1)/2} \setminus \cO$ and $\cH_2 = \cE_{(q-1)/2} \setminus \cO$. 
By Proposition~\ref{point-orbits} and Lemma~\ref{lemma:size}, the set 
%\[
%\left\{U, R_1,\dots, R_{q}, Q_0,\dots,Q_{q-1}, T_1, T_2\right\}
%\] 
%consists of the representatives of the $K$-orbits on points of $\PG(3, q^2) \setminus \Sigma$ and hence 
\[
\left\{U^\perp, R_1^\perp, \dots, R_{q}^\perp, Q_0^\perp, \dots, Q_{q-1}^\perp, T_1^\perp, T_2^\perp\right\}
\] 
forms a set of the representatives of the $K$-orbits on planes of $\PG(3, q^2)$ having in common with $\Sigma$ a Baer subline. 

\subsection{On the intersection of \texorpdfstring{$\cS_j$}{} and \texorpdfstring{$\cE_k$}{} with the lines in \texorpdfstring{$\PG(3,q^2)$}{}}

In this subsection we prove some results concerning the intersection of the hypersurfaces $\cS_j$, $j \in \{1, \dots, q\} \setminus \{(q+1)/2\}$ and $\cE_k$, $k \in \{0, \dots, q-1\} \setminus \{(q-1)/2\}$ with the lines of $\PG(3, q^2)$. Here and in the sequel we will use a number of times the fact that every point of $\PG(3, q^2) \setminus \Sigma$ lies on a unique extended subline of $\Sigma$. Let us mention some properties on the action of $K$ on the extended sublines of $\Sigma$ that can be deduced by direct computations (see also \cite[p. 343]{CoPa}).

\begin{lemma}\label{lemma:known}
The following hold.
\begin{itemize}
\item[i)] The group $K$ permutes in a unique orbit the extended sublines of $\Sigma$ that are either secant or external to $\cO$ and has two orbits $\cT_1$, $\cT_2$ both of size $(q+1)(q^2+1)/2$ on the extended sublines of $\Sigma$ that are tangent to $\cO$. 
\item[ii)] A line of $\cT_i$ has one point of $\cO$ and $q$ points of $\Sigma_i$, a secant extended subline has two points of $\cO$ and $(q-1)/2$ points with both $\Sigma_1$ and $\Sigma_2$, whereas an external extended subline has $(q+1)/2$ points with both $\Sigma_1$ and $\Sigma_2$.
\item[iii)] If $\ell \in \cT_1$, then $\ell^\perp$ belongs to $\cT_1$ if $q \equiv -1 \pmod{4}$ or to $\cT_2$ if $q \equiv 1 \pmod{4}$.
\item[iv)] The sets $\cE_0$ and $\cE_{q-1}$ consist of the points on the lines of $\cT_1$ and $\cT_2$, respectively, and not on $\Sigma\setminus \cO$.
\end{itemize}
\end{lemma} 

\begin{lemma}\label{property}
The following hold.
\begin{itemize}
\item[i)] If an extended subline of $\Sigma$ meets $\cH_1$ or $\cS_j \setminus \cO$, then it is secant to $\cO$ and has $q-1$ points in common with $\cH_1$ or $\cS_j \setminus \cO$, $j \in \{1, \dots, q\} \setminus \{(q+1)/2\}$.
\item[ii)] If an extended subline of $\Sigma$ meets $\cH_2$ or $\cE_k \setminus \cO$, then it is external to $\cO$ and has $q+1$ points in common with $\cH_2$ or $\cE_k$, $k \in \{1, \dots, q-2\} \setminus \{(q-1)/2\}$.
\end{itemize}
\end{lemma}
\begin{proof}
Consider the lines $\ell_1: X_2 = X_3 = 0$ and $\ell_2 : X_1 = X_4 = 0$. By Lemma~\ref{lemma:known} it is enough to observe that $\ell_1 \cap \cH_1$ or $\ell_1 \cap \cS_j$ is formed by
\begin{align*}
& \left\{\left(d^{2(q+1)}\xi^j, 0, 0, 1\right) \mid d \in \F_{q^2} \setminus \{0\}\right\} \cup \left\{\left(d^{2(q+1)}\xi^{-j}, 0, 0, 1\right) \mid d \in \F_{q^2} \setminus \{0\}\right\}
\end{align*}
and that $\ell_2 \cap \cH_2$ or $\ell_2 \cap \cE_k$ consists of
\begin{align*}
& \left\{\left(0, d^{2(q-1)}\delta^k, 1, 0\right) \mid d \in \F_{q^2} \setminus \{0\}\right\} \cup \left\{\left(0, d^{2(q-1)}\delta^{-k}, 1, 0\right) \mid d \in \F_{q^2} \setminus \{0\}\right\},
\end{align*}
where $\delta$ equals $\xi^k$ if $q \equiv 1 \pmod{4}$ or $\xi^{q-1+k}$, if $q \equiv -1 \pmod{4}$. 
\end{proof}

\begin{lemma}\label{property_j}
The following hold.
\begin{enumerate}
\item[i)] A line of $\PG(3, q^2)$ that is secant to $\cQ^+(3, q^2)$ and meets $\cO$ in exactly one point, has $(q-1)/2$ points in common with $\cS_j \setminus \cO$, $j \in \{1, \dots, q\} \setminus \{(q+1)/2\}$.
\item[ii)] A line of $\PG(3, q^2)$ that is tangent to $\cQ^+(3, q^2)$ and meets $\cO$ in exactly one point, has no point in common with $\cS_j \setminus \cO$, $j \in \{1, \dots, q\} \setminus \{(q+1)/2\}$.
\end{enumerate}
\end{lemma}
\begin{proof}
The group $K$ acts transitively on the $q^2(q^4-1)$ lines of $\PG(3, q^2)$ that are secant to $\cQ^+(3, q^2)$ and have one point in common with $\cO$. Indeed, the stabilizer of the line $\ell: X_2 = X_1 - X_3 = 0$ spanned by $(1,0,1,0)$ and $(0,0,0,1)$ is trivial. On the other hand, the point $(1,0,1,y)$ belongs to $\cS_j \setminus \cO$ if and only if 
\begin{align*}
& y = \xi^j x^2, \\
& (\xi^jx - \xi^{jq} x^q)(x-x^q) = 1.
\end{align*}
If $\xi^j = \xi_0^j + \i \xi_1^j$ and $x = x_0 + \i x_1$, the latter equation becomes $4s x_1(x_0\xi_1^j + x_1 \xi_0^j) = 1$, which is satisfied by $q-1$ couples $(x_0, x_1) \in \F_q$. Finally observe that $(x_0, x_1)$ and $(-x_0,-x_1)$ are solutions that give rise to the same $y$. This shows {\em i)}.

In order to prove {\em ii)}, let $\ell$ be a line of $\PG(3, q^2)$ that is tangent to $\cQ^+(3, q^2)$ at $P$, where $P \in \cO$. Then $\ell \subset P^\perp$, where $|P^\perp \cap \Sigma| = q^2+q+1$. By Lemma~\ref{property} $|P^\perp \cap (\cS_j \setminus \cO)| = 0$. 
\end{proof}

In \cite{CP}, the authors proved the following. 

\begin{lemma}\label{property_k1}
Let $\cL$ be the set of $(q+1)(q^2+1)$ lines of $\cH(3,q^2)$ meeting $\cO$ in one point. The following hold.
\begin{itemize}
\item[i)] A line of $\cH(3, q^2)$ has at most $(q^2+1)/2$ points in common with $\cH_1$. 
\item[ii)] A line of $\cH(3, q^2)$ not belonging to $\cL$ has at most $(q^2+1)/2$ points in common with $\cH_2$. 
\item[iii)] The lines of $\cL$ are contained in $\cH_2 \cup \cO$. Through each point of $\cH_2$ there pass two lines of $\cL$ and through each point of $\cO$ there pass $q+1$ lines of $\cL$. Moreover, if $\ell \in \cL$, then $|\ell \cap \Sigma| = 1$. 
\end{itemize}
\end{lemma}
%The next result shows that a similar property holds true for $\cE_k$.
\begin{lemma}\label{property_k}
If $k \in \{1, \dots, q-2\} \setminus \{(q-1)/2\}$, then there is a set $\cL_k$ of $(q+1)(q^2+1)$ lines of $\PG(3, q^2)$ contained in $\cE_k$ such that through each point of $\cE_k \setminus \cO$ there pass two lines of $\cL_k$ and through each point of $\cO$ there pass $q+1$ lines of $\cL_k$. Moreover, if $\ell \in \cL_k$, then $|\ell \cap \Sigma| = 1$.
\end{lemma}
\begin{proof}
If $P \in \cE_k \setminus \cO$, then by Lemma~\ref{property} {\em ii)}, there is an extended subline $\ell_2$ of $\Sigma$ such that $|\ell_2 \cap \cO| = 0$ and $P \in \ell_2$. Hence $|\ell_2^\perp \cap \cO| = 2$. We claim that the two lines obtained by joining $P$ with the two points $\ell_2^\perp \cap \cO$ are contained in $\cE_k$. If $q \equiv 1 \pmod{4}$ let $P = Q_k$. Then $Q_k \in \cE_k$ and it is easily seen that each of the two lines obtained by joining $Q_k$ with the two points $(1,0,0,0)$ and $(0,0,0,1)$ of $\cO$ is contained in $\cE_k\setminus\cO$ and meets $\Sigma$ exactly in one point of $\cO$. In particular each of these lines has $q^2$ points of $\cE_k \setminus \cO$. The case $q \equiv -1 \pmod{4}$ is analogous. Denote by $\cL_k$ the set of lines obtained in this way. By double counting the flags $(P_1, \ell)$, $P_1 \in \cE_k \setminus \cO$, $\ell \in \cL_k$, $P_1 \in \ell$, it follows that $|\cL_k| = (q+1)(q^2+1)$. Finally, by double counting the flags $(P_2, \ell)$, $P_2 \in \cO$, $\ell \in \cL_k$, $P_2 \in \ell$, we obtain that through each point of $\cO$ there pass $q+1$ lines of $\cL_k$.  
\end{proof}

\begin{lemma}\label{prop1}
If $\ell \in \cL_k$, then $\ell \subseteq P^\perp \iff P \in \cE_{q-1-k}$. Moreover, if $P \in \cE_{q-1-k}$, then $P^\perp$ contains precisely two lines of $\cL_k$. 
\end{lemma}
\begin{proof}
Let $P = (0, (-1)^{(q+1)/2}\xi^k, 1,0) \in \cE_{q-1-k}$ and $P' = (0, -(-1)^{(q+1)/2}\xi^k, 1,0) \in \cE_k$. We may assume that $\ell \in \cL_k$ is the line joining $(1,0,0,0)$ and $P'$. For the first part of the statement it is enough to observe that $\ell^\perp = \langle P, (0,0,0,1) \rangle \subset \cE_{q-1-k}$. To complete the proof, direct computations show that $\langle P', (1,0,0,0) \rangle$ and $\langle P', (0,0,0,1) \rangle$ are the lines of $\cL_k$ contained in the plane $P^\perp$.
\end{proof}

\begin{lemma}\label{property_k2}
Let $\ell$ be a line of $\PG(3, q^2)$. The following hold.
\begin{itemize}
\item[i)] $|\ell \cap \cS_j| \le 2q+2$, $j \in \{1, \dots, q\} \setminus \{(q+1)/2\}$, 
\item[ii)] if $\ell \notin \cL_k$, then $|\ell \cap \cE_k| \le 2q+2$, $k \in \{1, \dots, q-2\} \setminus \{(q-1)/2\}$, 
\end{itemize}
\end{lemma}
\begin{proof}
Let $\ell$ be a line of $\PG(3, q^2)$ such that $|\ell \cap \cS_j| \ge 4$. Since no three points of $\cO$ are collinear, the line $\ell$ has at least two points in common with $\cS_j \setminus \cO$. By Proposition~\ref{point-orbits}, we may assume that the points $R_j$ and $R_j^{A^q \otimes A}$ are in $\ell \cap \left(\cS_j \setminus \cO\right)$. The point of $\ell$
\begin{align*}
R_j^{A^q \otimes A} + \lambda R_j = \left( \xi^j a^{q+1} + b^{q+1}, \xi^j a^qc+b^qd, \xi^jac^q+bd^q, \xi^j c^{q+1} +d^{q+1} \right), \quad \lambda \in \F_{q^2}, 
\end{align*}
belongs to $\cS_j$ if and only if  
\begin{align}
& F(\lambda) - (\xi^{jq} + \xi^j) G(\lambda)^{\frac{q+1}{2}} = 0, \label{eq3}
\end{align}
where
\begin{align*}
F(\lambda) = & \lambda^{q+1} (\xi^j + \xi^{jq}) + \lambda^q (\xi^{j(q+1)}c^{q+1} + \xi^{jq} d^{q+1} + \xi^j a^{q+1} + b^{q+1}) && \\
  & + \lambda (\xi^{j(q+1)}c^{q+1} + \xi^{jq} a^{q+1} + \xi^j d^{q+1} + b^{q+1}) + (ad-bc)^{q+1} (\xi^j + \xi^{jq}), && \\
G(\lambda) = & \lambda^2 + \lambda (\xi^j c^{q+1} + a^{q+1} + d^{q+1} + \xi^{-j} b^{q+1}) + (ad-bc)^{q+1}. &&
\end{align*}
Eq. \eqref{eq3} implies 
\begin{align}
& F(\lambda)^2 - (\xi^{jq} + \xi^j)^2 G(\lambda)^{q+1} = \lambda^{2q+1} C_{2q+1} + \lambda^{2q} C_{2q} + \ldots = 0, \label{eq4}
\end{align}
with
\begin{align*}
%F(\lambda)^2 - (\xi^{jq} + \xi^j)^2 G(\lambda)^{q+1} = & \lambda^{2q+1} C_{2q+1} + \lambda^{2q} C_{2q} + \dots && \\
C_{2q+1} = & (\xi^{jq} - \xi^j) (\xi^j c^{q+1} - a^{q+1} + d^{q+1} - \xi^{-j} b^{q+1}), && \\
C_{2q} = & (\xi^{j(q+1)}c^{q+1} + \xi^{jq} d^{q+1} + \xi^j a^{q+1} + b^{q+1})^2 + (\xi^{jq} + \xi^j)^2 (ad-bc)^{q+1}. &&
\end{align*}
Since $\xi^j \ne \xi^{jq}$, if $\xi^j c^{q+1} + d^{q+1} \ne a^{q+1} + \xi^{-j} b^{q+1}$, then $C_{2q+1} \ne 0$ and eq. \eqref{eq4} has at most $2q+1$ solutions. Hence $|\ell \cap \cS_j| \le 2q+2$. If $\xi^j c^{q+1} + d^{q+1} = a^{q+1} + \xi^{-j} b^{q+1}$, then $C_{2q+1} = 0$ and 
\begin{align}
C_{2q} = (\xi^{jq} + \xi^j)^2 \left( (a^{q+1} + \xi^{-j} b^{q+1})^2 - (ad-bc)^{q+1} \right). \label{eq5} &&
\end{align}
Assume $C_{2q} = 0$. Since $\xi^j \ne - \xi^{jq}$ and $ad-bc \in \Box_{q^2}$, from eq. \eqref{eq5} we infer that $(a^{q+1} + \xi^{-j} b^{q+1})^2 \in \Box_{q}$, that is $a^{q+1} + \xi^{-j}b^{q+1} \in \F_q \iff b = 0$. In particular $a^{q+1} = d^{q+1}$ and $c = 0$. Hence $R_j^{A^q \otimes A} = R_j$.

Let $\ell$ be a line of $\PG(3, q^2)$, $q \equiv 1 \pmod{4}$, such that $|\ell \cap \cE_k| \ge 4$. As before, by Proposition~\ref{point-orbits}, we may assume that the points $Q_k$ and $Q_k^{A^q \otimes A}$ are in $\ell \cap \left(\cE_k \setminus \cO\right)$. In this case the point of $\ell$ given by 
\begin{align*}
Q_k^{A^q \otimes A} + \lambda Q_k = \left( \xi^k a^qb + b^qa, \xi^k a^qd+b^qc + \xi^k \lambda, \xi^k bc^q+ad^q+\lambda, \xi^k c^qd +d^qc \right), \quad \lambda \in \F_{q^2}, 
\end{align*}
belongs to $\cE_k$ if and only if  
\begin{align}
& F'(\lambda) - (\xi^{k(q+1)} + 1) G'(\lambda)^{\frac{q+1}{2}} = 0, \label{eq6}
\end{align}
where
\begin{align*}
F'(\lambda) = & - \lambda^{q+1} (\xi^{k(q+1)} + 1) - \lambda^q (\xi^{k(q+1)}a^{q}d + \xi^{kq} b^{q}c + \xi^k bc^{q} + ad^{q}) && \\
  & - \lambda (\xi^{k(q+1)}ad^{q} + \xi^{kq} b^{q}c + \xi^k bc^{q} + a^{q}d) - (ad-bc)^{q+1} (\xi^{k(q+1)} + 1), && \\
G'(\lambda) = & - \lambda^2 - \lambda (\xi^k bc^{q} + ad^{q} + a^{q}d + \xi^{-k} b^{q}c) - (ad-bc)^{q+1}. &&
\end{align*}
Eq. \eqref{eq6} implies 
\begin{align}
& F'(\lambda)^2 - (\xi^{k(q+1)} + 1)^2 G'(\lambda)^{q+1} = \lambda^{2q+1} C'_{2q+1} + \lambda^{2q} C'_{2q} + \ldots = 0, \label{eq7}
\end{align}
with
\begin{align*}
%F(\lambda)^2 - (\xi^{jq} + \xi^j)^2 G(\lambda)^{q+1} = & \lambda^{2q+1} C_{2q+1} + \lambda^{2q} C_{2q} + \dots && \\
C'_{2q+1} = & (\xi^{k(q+1)} - 1) (a^qd+\xi^{-k} b^qc - ad^{q} - \xi^{k} bc^{q}), && \\
C'_{2q} = & (\xi^{k(q+1)}a^qd + \xi^{kq} b^{q}c + \xi^k bc^{q} + ad^{q})^2 - (\xi^{k(q+1)} + 1)^2 (ad-bc)^{q+1}. &&
\end{align*}
Since $\xi^{k(q+1)} \ne 1$, if $a^qd+\xi^{-k} b^qc \ne ad^{q} + \xi^{k} bc^{q}$, then $C'_{2q+1} \ne 0$, eq. \eqref{eq7} has at most $2q+1$ solutions and therefore $|\ell \cap \cE_k| \le 2q+2$. If $a^qd+\xi^{-k} b^qc = ad^{q} + \xi^{k} bc^{q}$, then $C'_{2q+1} = 0$ and 
\begin{align}
C'_{2q} = (\xi^{k(q+1)} + 1)^2 \left( (ad^{q} + \xi^{k} bc^{q})^2 - (ad-bc)^{q+1} \right). \label{eq8} &&
\end{align}
Assume $C'_{2q} = 0$. Since $\xi^{k(q+1)} \ne - 1$ and $ad-bc \in \Box_{q^2}$, from eq. \eqref{eq8} it follows that $(ad^{q} + \xi^{k} bc^{q})^2 \in \Box_{q}$, that is $ad^{q} + \xi^{k} bc^{q} \in \F_q \iff ad^q \in \F_q$ and either $b = 0$ or $c = 0$. In these cases $F'(\lambda) - (\xi^{k(q+1)} + 1) G'(\lambda)^{\frac{q+1}{2}} \equiv 0$  and $\ell$ is one of the two lines of $\cL_k$ through $Q_k$. The case $q \equiv -1 \pmod{4}$ is analogous. 
\end{proof}

\subsection{Some auxiliaries surfaces and curves in \texorpdfstring{$\PG(3,q^2)$}{}}\label{hypersurfaces}
In this subsection we collect some technical results about certain algebraic sets $\cY_j$, defined below, which are necessary for the proof of Theorem~\ref{quasiHermitianTHM}.
%We consider matrices representing projectivities acting on the left.

For a fixed $i \in \{1, 2, \dots, q\} \setminus \{(q+1)/2\}$, we want to count the number of points of $\PG(3, q^2)$ such that 
\begin{align}
    & \xi^j X_1^{q+1} + X_2^{q+1} + \xi^{i+j} X_3^{q+1} + \xi^i X_4^{q+1} = 0, \label{pen} \\
    & X_1 X_4 - X_2 X_3 \in \Box_{q^2}, \label{hyp}
\end{align}
where $j \in \{1, 2, \dots, q\} \setminus \{(q+1)/2\}$. Let $\cX_j$ denote the set of points of $\PG(3, q^2)$ satisfying \eqref{pen} and let $\cY_j$ be the set of points of $\cX_j$ satisfying \eqref{hyp}. We want to determine the size of $\cY_j$, for $j \in \{1, 2, \dots, q\} \setminus \{(q+1)/2\}$. It can be easily seen that $\cQ^+(3, q^2)$ has no point in common with $\cX_j$.
\begin{lemma}
$|\cX_j \cap \cQ^+(3, q^2)| = 0$.
\end{lemma}
\begin{proof}
In order to prove that no point of $\cQ^+(3, q^2)$ lies on $\cX_j$, consider the point $(x_1y_1, x_1y_2, x_2y_1, x_2y_2)$, $x_1,x_2,y_1,y_2 \in \F_{q^2}$, $(x_1, x_2) \ne (0,0)$, $(y_1,y_2) \ne (0,0)$, belonging to $\cQ^+(3, q^2)$ and observe that it does not satisfies \eqref{pen}, otherwise $(\xi^j y_1^{q+1} + y_2^{q+1})(x_1^{q+1}+\xi^i x_2^{q+1}) = 0$, a contradiction. 
\end{proof}

A point of $\cX_j$ satisfies $\xi^{jq} X_1^{q+1} + X_2^{q+1} + \xi^{(i+j)q} X_3^{q+1} + \xi^{iq} X_4^{q+1} = 0$, that is, it lies on both Hermitian varieties  
\begin{align*}
    & H_1 : (\xi^j+\xi^{jq}) X_1^{q+1} + 2X_2^{q+1} + (\xi^{i+j}+\xi^{(i+j)q}) X_3^{q+1} + (\xi^i+\xi^{iq}) X_4^{q+1} = 0, \\
    & H_2: \i(\xi^j-\xi^{jq}) X_1^{q+1} + \i(\xi^{i+j}-\xi^{(i+j)q}) X_3^{q+1} + \i(\xi^i-\xi^{iq}) X_4^{q+1} = 0;
\end{align*}
where $\i^q=-\i$.
In fact $\cX_j$ is the base locus of the pencil  
\begin{align*}
    & \cP = \{H_1 + t  H_2 \mid t \in \F_q\} \cup \{H_2\}.
\end{align*}
With a slight abuse of language we will refer to the {\em rank} of a quadric (or of a Hermitian variety) as the rank of a symmetric (or of a Hermitian) matrix defining it. Set $\xi^k = \xi_0^k + \i \xi_1^k$, for some $\xi_0^k, \xi_1^k \in \F_q$. Then 
\begin{align*}
& \rk (H_2) = \rk (\diag(\xi_1^j, 0, \xi_0^j\xi_1^i+\xi_0^i\xi_1^j, \xi_1^i)), \\
& \rk (H_1+ t\i H_2) = \rk (\diag(\xi_0^j + ts \xi_1^j, 1, \xi_0^i\xi_0^j+s\xi_i^i\xi_1^j+st(\xi_0^j\xi_1^i+\xi_0^i\xi_1^j), \xi_0^i + st\xi_1^i)).
\end{align*} 

\begin{lemma}\label{pencil1}
If $j \in \{i, q+1-i\}$, the pencil $\cP$ has one element of rank two, two members of rank three and $q-2$ of rank four, whereas if $j \in \{1, 2, \dots, q\} \setminus \{i, (q+1)/2, q+1-i\}$, $\cP$ has four members of rank three and $q-3$ of rank four. 
\end{lemma}
\begin{proof}
Note that $\xi_1^j \ne 0$, if $j \in \{1, 2, \dots, q\} \setminus \{(q+1)/2\}$ and 
\[
\xi_0^j\xi_1^i+\xi_0^i\xi_1^j = 0 \iff \xi^{i+j} \in \F_q \iff \xi^j = \xi^{q+1-i}.
\]
Hence 
\[
\rk(H_2) = 
\begin{cases}
2 & \mbox{ if } j = q+1-i, \\
3 & \mbox{ otherwise. }
\end{cases}
\]
On the other hand, there is no $t \in \F_q$ such that 
\begin{align*}
& \xi_0^j + st \xi_1^j = 0, \\
& \xi_0^i\xi_0^j+s\xi_i^i\xi_1^j+st(\xi_0^j\xi_1^i+\xi_0^i\xi_1^j) = 0,
\end{align*}
otherwise  
\[
 0 = \xi_0^i\xi_0^j+s\xi_i^i\xi_1^j+st(\xi_0^j\xi_1^i+\xi_0^i\xi_1^j) = - \frac{\xi_1^i}{\xi_1^j} \xi^{j(q+1)}, 
\]
a contradiction. Moreover
\[
\xi_0^i+st\xi_1^j = \xi_0^i+st\xi_1^i = 0 \iff \xi^i = \xi^j.
\]
Therefore 
\[
\rk(H_1+ t \i H_2) = 2 \iff i = j \mbox{ and } t = - \frac{\xi_0^j}{s\xi_1^j}.
\]
The result follows.
\end{proof}

Set $X_i = Y_i + \i Z_i$ and let $\PG(7, q)$ be the seven-dimensional projective space equipped with projective coordinates $(Y_1, Z_1, Y_2, Z_2, Y_3, Z_3, Y_4, Z_4)$. Let $\phi$ denote the {\em field reduction map} which sends the point of $\PG(3, q^2)$ represented by $(x_1,x_2,x_3,x_4)$ to the line of $\PG(7, q)$ spanned by the points $(s z_1, y_1, s z_2, y_2, s z_3, y_3, s z_4, y_4)$ and $(y_1, z_1, y_2, z_2, y_3, z_3, y_4, z_4)$. The map $\phi$ sends $\cQ^+(3, q^2)$ to the base locus of the pencil $\cP_1$ of hyperbolic quadrics of $\PG(7, q)$ generated by 
\begin{align*}
& \cQ_1: Y_1 Z_4 + Y_4 Z_1 - Y_2Z_3 - Y_3Z_2 = 0, \\
& \cQ_2: Y_1 Y_4 + sZ_1 Z_4 - Y_2 Y_3 - sZ_2Z_3 = 0.  
\end{align*}
A member of $\cP$ having rank either four, or three or two is sent by $\phi$ either to a hyperbolic quadric, or to a cone with vertex a line and base a $\cQ^-(5, q)$ or to a cone with vertex a solid and base a $\cQ^+(3, q)$, respectively. In particular, the images of the elements of $\cP$ form the pencil $\cP_2$ of quadrics of $\PG(7, q)$ generated by
\begin{align*}
& H_1^\phi: \xi_0^j (Y_1^2-sZ_1^2) + Y_2^2-sZ_2^2 + (\xi_0^i\xi_0^j+s\xi_i^i\xi_1^j) (Y_3^2-sZ_3^2) + \xi_0^i (Y_4^2-sZ_4^2) = 0, \\
& H_2^\phi: \xi_1^j (Y_1^2-sZ_1^2) + (\xi_0^i\xi_1^j+\xi_0^j\xi_1^i) (Y_3^2-sZ_3^2) + \xi_1^i (Y_4^2-sZ_4^2) = 0.  
\end{align*}
Let $\cN$ be the net of quadrics of $\PG(7, q)$ containing $\cQ_1$, $H_1^\phi$, $H_2^\phi$ and denote by $\cB$ the base locus of $\cN$. We claim that $|\cY_j| = 2|\cB|$.   

\begin{lemma}\label{size}
$|\cY_j| = 2|\cB|$.
\end{lemma}
\begin{proof}
On the one hand, if $P$ is a point of $H_1 \cap H_2$, then the line $P^\phi$ is contained in $H_1^\phi \cap H_2^\phi$. On the other hand, let $P$ be a point not belonging to $\cQ^+(3, q^2)$ and such that the quadratic form defining $\cQ^+(3, q^2)$, evaluated in $P$, is a square in $\F_{q^2}$; since the group $\PSO^+(4, q^2)$ stabilizing $\cQ^+(3, q^2)$ is transitive on these square points and $\PSO^+(4, q^2)$ can be embedded as a subgroup of the $\PGO^+(8, q)$ leaving invariant $\cQ_1$, we may assume, without loss of generality, that $P$ is represented by $(1,0,0,1)$. Then $P^\phi = \langle (1,0,0,0,0,0,1,0), (0,1,0,0,0,0,0,1) \rangle$, and the line $P^\phi$ is secant to $\cQ_1$.   
\end{proof}

Lemma~\ref{pencil1} implies the following.

\begin{lemma}\label{infty-net}
If $j \in \{i, q+1-i\}$, the pencil $\cP_2$ has a cone with vertex a solid and base a $\cQ^+(3, q)$, two cones with vertex a line and base a $\cQ^-(5, q)$ and $q-2$ hyperbolic quadrics, whereas if $j \in \{1, 2, \dots, q\} \setminus \{i, (q+1)/2, q+1-i\}$, $\cP_2$ has four cones with vertex a line and base a $\cQ^-(5, q)$ and $q-3$ hyperbolic quadrics. 
\end{lemma}

In order to determine the size of $\cB$, we study the rank of the quadric $\cQ_1 + \alpha \cH_1^\phi + \beta \cH_2^\phi \in \cN \setminus \cP_2$, $\alpha, \beta \in \F_q$. We have that $\rk(\cQ_1 + \alpha \cH_1^\phi + \beta \cH_2^\phi) = \rk(M_{\alpha, \beta})$, where
\[
M_{\alpha, \beta} = 
\begin{pmatrix}
A & 0 & 0 & 0 & 0 & 0 & 0 & 1 \\
0 & -sA & 0 & 0 & 0 & 0 & 1 & 0 \\
0 & 0 & \alpha & 0 & 0 & -1 & 0 & 0 \\
0 & 0 & 0 & -s \alpha & -1 & 0 & 0 & 0 \\
0 & 0 & 0 & -1 & C & 0 & 0 & 0 \\
0 & 0 & -1 & 0 & 0 & -sC & 0 & 0 \\
0 & 1 & 0 & 0 & 0 & 0 & D & 0 \\
1 & 0 & 0 & 0 & 0 & 0 & 0 & -sD \\
\end{pmatrix},
\]
with $A = \alpha \xi_0^j + \beta \xi_1^j$, $C = \alpha(\xi_0^i\xi_0^j+s\xi_i^i\xi_1^j)+ \beta(\xi_0^j\xi_1^i+\xi_0^i\xi_1^j)$, $D = \alpha \xi_0^i + \beta \xi_1^i$. 

\begin{lemma} \label{affine-net}
$\cQ_1 + \alpha \cH_1^\phi + \beta \cH_2^\phi$ is either a hyperbolic quadric or a cone with vertex a line and base a $\cQ^-(5, q)$.
\end{lemma}
\begin{proof}
Observe that 
\[
\det(M_{\alpha, \beta}) = (sAD+1)^2(s \alpha C+1)^2
\]
and hence if $\det(M_{\alpha, \beta}) \ne 0$, then $\cQ_1 + \alpha \cH_1^\phi + \beta \cH_2^\phi$ is a hyperbolic quadric, see \cite[Theorem 1.2]{HT}. Each of the minors of order seven of $M_{\alpha, \beta}$ have both $sAd+1$ and $s \alpha C + 1$ as factors. The two minors of order six of $M_{\alpha, \beta}$ obtained by deleting the third and the fifth row and the fourth and sixth column or by deleting the first and the seventh row and the second and last column are $-(aAD + 1)^2$ and $-(s \alpha C + 1)^2$. If 
\begin{align}
& 0 = s A D + 1 = s (\alpha \xi_0^j + \beta \xi_1^j) (\alpha \xi_0^i + \beta \xi_1^i) + 1, \label{eq1} \\
& 0 = s \alpha C + 1 = s \alpha^2(\xi_0^i\xi_0^j+s\xi_i^i\xi_1^j) + \alpha \beta(\xi_0^j\xi_1^i+\xi_0^i\xi_1^j) + 1 \label{eq2}, 
\end{align}
then $s (s\alpha^2-\beta^2) \xi_1^i \xi_1^j = 0$, a contradiction. Hence $\rk(M_{\alpha, \beta}) \in \{6, 8\}$. To complete the proof, assume that $\rk(M_{\alpha, \beta}) = 6$ and $s \alpha C + 1 = 0$, then the five-dimensional projective space $Y_2 = Z_2 = 0$ does not contain the vertex and intersects $\cQ_1 + \alpha \cH_1^\phi + \beta \cH_2^\phi$ in the $\cQ^-(5, q)$ given by $A Y_1^2 + 2 Y_1 Z_4 -sD Z_4^2 -sA Z_1^2 +2 Z_1 Y_4 + D Y_4^2 + \alpha^{-1} Z_3^2 - (s \alpha)^{-1} Y_3^2 = 0$. Analogously if $\rk(M_{\alpha, \beta}) = 6$ and $s AD + 1 = 0$
\end{proof}

\begin{prop}\label{net}
The following hold.
\begin{itemize}
\item If $j \in \{1, 2, \dots, q\} \setminus \{i, (q+1)/2, q+1-i\}$, then $\cN$ has $2(q+1)$ cones with vertex a line and base a $\cQ^-(5, q)$ and $q^2-q-1$ hyperbolic quadrics.
\item If $j = i$ and $q \equiv 1 \pmod{4}$ or $j = q+1 - i$ and $q \equiv -1 \pmod{4}$, then $\cN$ has one cone with vertex a solid and base a $\cQ^+(3, q)$, $3q+1$ cones with vertex a line and base a $\cQ^-(5, q)$ and $q^2-2q-1$ hyperbolic quadrics.
\item If $j = i$ and $q \equiv -1 \pmod{4}$ or $j = q+1 - i$ and $q \equiv 1 \pmod{4}$, then $\cN$ has one cone with vertex a solid and base a $\cQ^+(3, q)$, $q+1$ cones with vertex a line and base a $\cQ^-(5, q)$ and $q^2-1$ hyperbolic quadrics.
\end{itemize}
\end{prop}
\begin{proof}
From the proof of Lemma~\ref{affine-net}, the quadric $\cQ_1 + \alpha \cH_1^\phi + \beta \cH_2^\phi$ has rank six if and only if $(\alpha, \beta) \in \F_q^2$ satisfies equation \eqref{eq1} or equation \eqref{eq2}. The result will follow by combining the number of these couples with Lemma~\ref{infty-net}. 

If $j \ne i$ and $j \ne q+1- i$, then each of the equations \eqref{eq1}, \eqref{eq2}, has $q-1$ solutions $(\alpha, \beta) \in \F_q^2$.   

If $j = i$, then equation \eqref{eq2} has $q-1$ solutions and equation \eqref{eq1} becomes $sA^2+1 = 0$, which has $2q$ or none solution according as $q \equiv -1$ or $1 \pmod{4}$.

If $j = q+1-i$, then equation \eqref{eq1} has $q-1$ solutions and equation \eqref{eq2} becomes $0 = s\alpha^2(\xi_0^i\xi_0^j+s\xi_i^i\xi_1^j)+1 = s \alpha^2 \xi^{q+1} + 1$, which has $2q$ or none solution according as $q \equiv 1$ or $-1 \pmod{4}$.
\end{proof}

\begin{prop}\label{solution}
\begin{align*}
|\cY_j| = 
	\begin{cases} 
		\frac{(q+1)(q^3+q^2-q+1)}{2}, & \mbox{ if } j = i \mbox{ and } q \equiv -1 \pmod{4} \mbox{ or } j = q + 1 - i \mbox{ and } q \equiv 1 \pmod{4}, \\ 
		\frac{(q^2-1)^2}{2}, & \mbox{ otherwise}. %j = 1 \mbox{ and } q \equiv 1 \pmod{4} \mbox{ or } j = q \mbox{ and } q \equiv -1 \pmod{4} \mbox{ or } j \in \{2, \dots, q-1\} \setminus \{(q+1)/2\}, \\
	\end{cases}
\end{align*}
\end{prop}
\begin{proof}
A point of $\PG(7, q)$ belongs to all or $q+1$ members of $\cN$, according as it lies on $\cB$ or not. Hence $|\cB|$ can be computed via
\begin{align*}
    & \sum_{\cQ \in \cN} |\cQ| = \left((q+1)(q^2+1)(q^4+1) - |\cB|\right) (q+1) + |\cB| (q^2+q+1).
\end{align*}
The result now follows from Proposition~\ref{net} and Lemma~\ref{size}.
\end{proof}

\subsection{The construction}

In the first part of this subsection, we will compute the point-orbit distributions of the planes under the action of $K$, namely the number of points of the $K$-orbits on points of $\PG(3,q^2)$ lying on a given plane of $\PG(3, q^2)$. Then we provide a new construction of quasi-Hermitian surfaces of $\PG(3, q^2)$, $q$ odd.

\begin{lemma}
If $P \in \cO$, then
\begin{itemize}
\item[] $|P^\perp \cap \cO| = 1$, 
\item[] $|P^\perp \cap \Sigma_1| = |P^\perp \cap \Sigma_2| = \frac{q^2+q}{2}$, 
\item[] $|P^\perp \cap (\cQ^+(3, q^2) \setminus \cO)| = 2q^2$, 
\item[]  $|P^\perp \cap \cH_1| = |P^\perp \cap (\cS_j \setminus \cO)| = 0$, $j \in \{1, \dots, q\} \setminus \left\{\frac{q+1}{2}\right\}$, 
\item[]  $|P^\perp \cap \cH_2| = |P^\perp \cap (\cE_k \setminus \cO)| = q^3+q^2$, $k \in \{1, \dots, q-2\} \setminus \left\{\frac{q-1}{2}\right\}$, 
\item[]  $|P^\perp \cap (\cE_0 \setminus \Sigma)| = |P^\perp \cap (\cE_{q-1} \setminus \Sigma)| = \frac{q^3-q}{2}$. 
\end{itemize} 
\end{lemma}
\begin{proof}
If $P \in \cO$, then $P^\perp \cap \Sigma$ is a Baer subplane and $|P^\perp \cap \cO| = 1$. In particular $P^\perp \cap \Sigma$ has $(q+1)/2$ extended sublines of $\cT_i$, $i = 1,2$, and $q^2$ extended sublines of $\Sigma$ that are external to $\cO$, see \cite[p. 343]{CoPa}. The results now follow from the fact that an extended subline that is tangent or external to $\cO$ has one or two points in common with $\cQ^+(3, q^2)$ and by Lemma~\ref{lemma:known} and Lemma~\ref{property}.
\end{proof}

\begin{lemma}
If $P \in \Sigma_1 \setminus \cO$, then
\begin{itemize}
\item[] $|P^\perp \cap \cO| = q+1$, 
\item[] $|P^\perp \cap \Sigma_1| = \frac{q^2+(-1)^{\frac{q+1}{2}}q}{2}$, $|P^\perp \cap \Sigma_2| = \frac{q^2-(-1)^{\frac{q+1}{2}}q}{2}$, 
\item[] $|P^\perp \cap (\cQ^+(3, q^2) \setminus \cO)| = q^2-q$, 
\item[] $|P^\perp \cap \cH_1| = |P^\perp \cap (\cS_j \setminus \cO)| = \frac{q^3-q}{2}$, $j \in \{1, \dots, q\} \setminus \left\{\frac{q+1}{2}\right\}$, 
\item[] $|P^\perp \cap \cH_2| = |P^\perp \cap (\cE_k \setminus \cO)| = \frac{q^3-q}{2}$, $k \in \{1, \dots, q-2\} \setminus \left\{\frac{q-1}{2}\right\}$, 
\item[] $|P^\perp \cap (\cE_0 \setminus \Sigma)| = \frac{q^3-q}{2}$, $|P^\perp \cap (\cE_{q-1} \setminus \Sigma)| = 0$, if $q \equiv -1 \pmod{4}$, 
\item[] $|P^\perp \cap (\cE_0 \setminus \Sigma)| = 0$, $|P^\perp \cap (\cE_{q-1} \setminus \Sigma)| =  \frac{q^3-q}{2}$, if $q \equiv 1 \pmod{4}$. 
\end{itemize}
If $P \in \Sigma_2 \setminus \cO$, then
\begin{itemize}
\item[] $|P^\perp \cap \cO| = q+1$, 
\item[] $|P^\perp \cap \Sigma_1| = \frac{q^2-(-1)^{\frac{q+1}{2}}q}{2}$, $|P^\perp \cap \Sigma_2| = \frac{q^2+(-1)^{\frac{q+1}{2}}q}{2}$, 
\item[] $|P^\perp \cap (\cQ^+(3, q^2) \setminus \cO)| = q^2-q$, 
\item[] $|P^\perp \cap \cH_1| = |P^\perp \cap (\cS_j \setminus \cO)| = \frac{q^3-q}{2}$, $j \in \{1, \dots, q\} \setminus \left\{\frac{q+1}{2}\right\}$, 
\item[] $|P^\perp \cap \cH_2| = |P^\perp \cap (\cE_k \setminus \cO)| = \frac{q^3-q}{2}$, $k \in \{1, \dots, q-2\} \setminus \left\{\frac{q-1}{2}\right\}$, 
\item[] $|P^\perp \cap (\cE_0 \setminus \Sigma)| = \frac{q^3-q}{2}$, $|P^\perp \cap (\cE_{q-1} \setminus \Sigma)| = 0$, if $q \equiv 1 \pmod{4}$, 
\item[] $|P^\perp \cap (\cE_0 \setminus \Sigma)| = 0$, $|P^\perp \cap (\cE_{q-1} \setminus \Sigma)| =  \frac{q^3-q}{2}$, if $q \equiv -1 \pmod{4}$. 
\end{itemize}
\end{lemma}
\begin{proof}
If $P \in \Sigma_i$, $i = 1,2$, then $P^\perp \cap \Sigma$ is a Baer subplane. Such a subplane meet $\cO$ in a Baer conic. Moreover, $P^\perp$ has $(q^2-q)/2$ extended sublines of $\Sigma$ that are external to $\cO$, $(q^2+q)/2$ extended sublines that are secant to $\cO$ and $q+1$ extended sublines of $\cT_1$ if $P \in \Sigma_2$ and $q \equiv 1 \pmod{4}$ or $P \in \Sigma_1$ and $q \equiv -1 \pmod{4}$, wheres it has $q+1$ extended sublines of $\cT_2$ if $P \in \Sigma_1$ and $q \equiv 1 \pmod{4}$ or $P \in \Sigma_2$ and $q \equiv -1 \pmod{4}$. The results now follow by Lemma~\ref{lemma:known} and Lemma~\ref{property}.
\end{proof}

\begin{lemma}
If $P \in \cQ^+(3, q^2) \setminus \cO$, then
\begin{itemize}
\item[] $|P^\perp \cap \cO| = 2$, 
\item[] $|P^\perp \cap \Sigma_1| = |P^\perp \cap \Sigma_2| = \frac{q-1}{2}$, 
\item[] $|P^\perp \cap (\cQ^+(3, q^2) \setminus \cO)| = 2q^2-1$, 
\item[] $|P^\perp \cap \cH_1| = |P^\perp \cap (\cS_j \setminus \cO)| = \frac{(q-1)(q^2+1)}{2}$, $j \in \{1, \dots, q\} \setminus \left\{\frac{q+1}{2}\right\}$, 
\item[] $|P^\perp \cap \cH_2| = |P^\perp \cap (\cE_k \setminus \cO)| = \frac{(q+1)(q^2-1)}{2}$, $k \in \{1, \dots, q-2\} \setminus \left\{\frac{q-1}{2}\right\}$, 
\item[] $|P^\perp \cap (\cE_0 \setminus \Sigma)| = |P^\perp \cap (\cE_{q-1} \setminus \Sigma)| = \frac{q^3-q}{2}$. 
\end{itemize} 
\end{lemma}
\begin{proof}
If $P \in \cQ^+(3, q^2) \setminus \cO$, then $P^\perp \cap \Sigma = r$ is a Baer subline, $|r \cap \cO| = 2$ and $|P^\perp \cap \cQ^+(3, q^2)| = 2q^2+1$. By Lemma~\ref{lemma:known}, $|P^\perp \cap \Sigma_1| = |P^\perp \cap \Sigma_2| = (q-1)/2$. There are $q+1+(q^2-1)/2$ lines of $\cT_i$ intersecting $r$ in one point, whereas $(q^3-q)/2$ of them meet $P^\perp$ in a point not of $r$. Hence $|P^\perp \cap (\cE_0 \setminus \Sigma)| = |P^\perp \cap (\cE_{q-1} \setminus \Sigma)| = (q^3-q)/2$ by Lemma~\ref{lemma:known}. Similarly, since by Lemma~\ref{prop1} no line of $\cL$ or of $\cL_k$ is contained in $P^\perp$, there are $2(q+1)$ lines of $\cL$ or of $\cL_k$ intersecting $r$ in one point and $(q+1)(q^2-1)$ of them meet $P^\perp$ in a point not belonging to $r$. Therefore $|P^\perp \cap \cH_2| = |P^\perp \cap (\cE_k \setminus \cO)| = (q+1)(q^2-1)/2$. Denote by $\tilde{r}$ the unique line of $\PG(3, q^2)$ containing $r$. Let $V$ be a point of $r \cap \cO$. Among the $q^2$ lines of $P^\perp$ through $V$ distinct from $\tilde{r}$, there is one line of $\cQ^+(3, q^2)$ and $q^2-1$ secants to $\cQ^+(3, q^2)$ and intersecting $\cO$ precisely in $V$. Hence, by Lemma~\ref{property} and Lemma~\ref{property_j}, it follows that $|P^\perp \cap (\cS_j \setminus \cO)| = (q-1)(q^2+1)/2$.
\end{proof}

\begin{lemma}
If $P \in \cH_2$, then
\begin{itemize}
\item[] $|P^\perp \cap \cO| = 2$, 
\item[] $|P^\perp \cap \Sigma_1| = |P^\perp \cap \Sigma_2| = \frac{q-1}{2}$, 
\item[] $|P^\perp \cap (\cQ^+(3, q^2) \setminus \cO)| = q^2-1$, 
\item[] $|P^\perp \cap \cH_1| = |P^\perp \cap (\cS_j \setminus \cO)| = \frac{(q-1)(q^2+1)}{2}$, $j \in \{1, \dots, q\} \setminus \left\{\frac{q+1}{2}\right\}$, 
\item[] $|P^\perp \cap \cH_2| =  \frac{(q+1)(q^2-1)}{2} + q^2$, $|P^\perp \cap (\cE_k \setminus \cO)| = \frac{(q+1)(q^2-1)}{2}$, $k \in \{1, \dots, q-2\} \setminus \left\{\frac{q-1}{2}\right\}$, 
\item[] $|P^\perp \cap (\cE_0 \setminus \Sigma)| = |P^\perp \cap (\cE_{q-1} \setminus \Sigma)| = \frac{q^3-q}{2}$. 
\end{itemize} 
\end{lemma}
\begin{proof}
If $P \in \cH_2$, then $P^\perp \cap \Sigma = r$ is a Baer subline, $|r \cap \cO| = 2$ and $|P^\perp \cap \cQ^+(3, q^2)| = q^2+1$. By Lemma~\ref{lemma:known}, $|P^\perp \cap \Sigma_1| = |P^\perp \cap \Sigma_2| = (q-1)/2$. As in the previous case, there are $q+1+(q^2-1)/2$ lines of $\cT_i$ intersecting $r$ in one point, whereas $(q^3-q)/2$ of them meet $P^\perp$ in a point not of $r$ and hence $|P^\perp \cap (\cE_0 \setminus \Sigma)| = |P^\perp \cap (\cE_{q-1} \setminus \Sigma)| = (q^3-q)/2$ by Lemma~\ref{lemma:known}. By \eqref{unitary} and Lemma~\ref{property_k1}, precisely two lines of $\cL$ are contained in $P^\perp$, whereas by Lemma~\ref{prop1} no line of $\cL_k$ lies in $P^\perp$. In this case $|P^\perp \cap \cH_2| = (q+1)(q^2-1)/2 + q^2$ and $|P^\perp \cap (\cE_k \setminus \cO)| = (q+1)(q^2-1)/2$. Hence $|P^\perp \cap \cH_1| = q^3+q^2+1 - (q+1)(q^2-1)/2 -q^2 - 2 = (q-1)(q^2+1)/2$. Denote by $\tilde{r}$ the unique line of $\PG(3, q^2)$ containing $r$. Let $V$ be a point of $r \cap \cO$. Among the $q^2$ lines of $P^\perp$ through $V$ distinct from $\tilde{r}$, there is one line that is tangent to $\cQ^+(3, q^2)$ and $q^2-1$ secants to $\cQ^+(3, q^2)$ and intersecting $\cO$ precisely in $V$. Hence, by Lemma~\ref{property} and Lemma~\ref{property_j}, it follows that $|P^\perp \cap (\cS_j \setminus \cO)| = (q-1)(q^2+1)/2$.
\end{proof}

\begin{lemma}
If $P \in \cE_i \setminus \cO$, $i \in \{1, \dots, q-2\} \setminus \left\{\frac{q-1}{2}\right\}$ then
\begin{itemize}
\item[] $|P^\perp \cap \cO| = 2$, 
\item[] $|P^\perp \cap \Sigma_1| = |P^\perp \cap \Sigma_2| = \frac{q-1}{2}$, 
\item[] $|P^\perp \cap (\cQ^+(3, q^2) \setminus \cO)| = q^2-1$, 
\item[] $|P^\perp \cap \cH_1| = |P^\perp \cap (\cS_j \setminus \cO)| = \frac{(q-1)(q^2+1)}{2}$, $j \in \{1, \dots, q\} \setminus \left\{\frac{q+1}{2}\right\}$, 
\item[] $|P^\perp \cap \cH_2| = |P^\perp \cap (\cE_k \setminus \cO)| = \frac{(q+1)(q^2-1)}{2}$, $k \in \{1, \dots, q-2\} \setminus \left\{\frac{q-1}{2}\right\}$, with $k \ne q-1-i$,
\item[] $|P^\perp \cap (\cE_{q-1-i} \setminus \cO)| = \frac{(q+1)(q^2-1)}{2}+q^2$,
\item[] $|P^\perp \cap (\cE_0 \setminus \Sigma)| = |P^\perp \cap (\cE_{q-1} \setminus \Sigma)| = \frac{q^3-q}{2}$. 
\end{itemize} 
\end{lemma}
\begin{proof}
If $P \in \cE_i \setminus \cO$, then by repeating the same arguments used in the previous lemma we obtain $|P^\perp \cap \cO| = 2$, $|P^\perp \cap \cQ^+(3, q^2)| = q^2+1$, $|P^\perp \cap \Sigma_1| = |P^\perp \cap \Sigma_2| = (q-1)/2$, $|P^\perp \cap (\cE_0 \setminus \Sigma)| = |P^\perp \cap (\cE_{q-1} \setminus \Sigma)| = (q^3-q)/2$ and $|P^\perp \cap (\cS_j \setminus \cO)| = (q-1)(q^2+1)/2$. No line of $\cL$ is contained in $P^\perp$ since $|P^\perp \cap \cH(3, q^2)| = q^3+1$, and by Lemma~\ref{prop1} two or no line of $\cL_k$ are in $P^\perp$, according as $k = q-1-i$ or not. Therefore $|P^\perp \cap \cH_2| = |P^\perp \cap (\cE_k \setminus \cO)| = (q+1)(q^2-1)/2$, $k \ne q-1-i$, $|P^\perp \cap (\cE_{q-1-k} \setminus \cO)| = (q+1)(q^2-1)/2+q^2$ and $|P^\perp \cap \cH_1| = (q-1)(q^2+1)/2$. 
\end{proof}

\begin{lemma}
If $P \in \cS_i \setminus \cO$, $i \in \{1, \dots, q\} \setminus \left\{\frac{q+1}{2}\right\}$, then
\begin{itemize}
\item[] $|P^\perp \cap \cO| = 0$, 
\item[] $|P^\perp \cap \Sigma_1| = |P^\perp \cap \Sigma_2| = \frac{q+1}{2}$, 
\item[] $|P^\perp \cap (\cQ^+(3, q^2) \setminus \cO)| = q^2+1$, 
\item[] $|P^\perp \cap \cH_1| = |P^\perp \cap (\cS_{j} \setminus \cO)| = \frac{(q-1)(q^2-1)}{2}$, $j \in \{1, \dots, q\} \setminus \left\{\frac{q+1}{2}\right\}$, with $j \ne q+1-i$ if $q \equiv -1\pmod{4}$, or $j \ne i$ if $q \equiv 1\pmod{4}$, 
\item[] $|P^\perp \cap (\cS_{q+1-i} \setminus \cO)| = \frac{(q-1)(q^2-1)}{2}+q^2$, if $q \equiv -1\pmod{4}$, 
\item[] $|P^\perp \cap (\cS_{i} \setminus \cO)| = \frac{(q-1)(q^2-1)}{2}+q^2$, if $q \equiv 1\pmod{4}$, 
\item[] $|P^\perp \cap \cH_2| = |P^\perp \cap (\cE_k \setminus \cO)| = \frac{(q+1)(q^2+1)}{2}$, $k \in \{1, \dots, q-2\} \setminus \left\{\frac{q-1}{2}\right\}$, 
\item[] $|P^\perp \cap (\cE_0 \setminus \Sigma)| = |P^\perp \cap (\cE_{q-1} \setminus \Sigma)| = \frac{q^3-q}{2}$. 
\end{itemize} 
\end{lemma}
\begin{proof}
If $P \in \cS_i \setminus \cO$, then $P^\perp \cap \Sigma = r$ is a Baer subline, $|r \cap \cO| = 0$ and $|P^\perp \cap \cQ^+(3, q^2)| = q^2+1$. By Lemma~\ref{lemma:known}, $|P^\perp \cap \Sigma_1| = |P^\perp \cap \Sigma_2| = (q+1)/2$. There are $(q+1)^2/2$ lines of $\cT_i$ intersecting $r$ in one point, whereas $(q^3-q)/2$ of them meet $P^\perp$ in a point not of $r$. Hence $|P^\perp \cap (\cE_0 \setminus \Sigma)| = |P^\perp \cap (\cE_{q-1} \setminus \Sigma)| = (q^3-q)/2$. No line of $\cL$ is contained in $P^\perp$ since $|P^\perp \cap \cH(3, q^2)| = q^3+1$ and by Lemma~\ref{prop1} no line of $\cL_k$ is contained in $P^\perp$. Moreover no line of $\cL$ or of $\cL_k$ intersects $r$. Therefore $|P^\perp \cap \cH_2| = |P^\perp \cap (\cE_k \setminus \cO)| = (q+1)(q^2+1)/2$ and hence $|P^\perp \cap \cH_1| = (q-1)(q^2-1)/2$. Assume now that $P = R_{i} = (\xi^{i}, 0, 0, 1) \in \cS_{i}$. In order to determine $|P^\perp \cap (\cS_{j} \setminus \cO)|$, we count the projectivities of $K$ mapping the point $R_{j} = (\xi^{j}, 0, 0, 1) \in \cS_{j} \setminus \cO$ to the plane $P^\perp: X_1 + \xi^{i} X_4 = 0$. Let $g \in K$ be induced by $A^q \otimes A$, where $A = \begin{pmatrix} a & b \\ c & d \end{pmatrix}$. Then $R_{j}^g \in P^\perp$ if and only if $(a, b, c, d)$, considered as a point of $\PG(3, q^2)$ fulfills both \eqref{pen} and \eqref{hyp}. By Proposition~\ref{solution}, there are either $(q^2-1)^2/2$ or $(q+1)(q^3+q^2-q+1)/2$ of these projectivities. On the other hand, if $R$ is a point of $P^\perp \cap \cS_j$ and $\cZ$ denotes the set of projectivities sending $R_j$ to $R$, then $\cZ g^{-1} = K_{R_j}$, whenever $g \in \cZ$. By Lemma~\ref{lemma:stab} we have that $|K_{R_j}| = q+1$, and hence $|P^\perp \cap (\cS_j \setminus \cO)|$ equals $(q-1)(q^2-1)/2+q^2$ if $j = i$ and $q \equiv -1 \pmod{4}$ or $j = q + 1 - i$ and $q \equiv 1 \pmod{4}$, whereas $|P^\perp \cap (\cS_j \setminus \cO)| = (q-1)(q^2-1)/2$ otherwise.
\end{proof}

\begin{lemma}
If $P \in \cH_1$, then
\begin{itemize}
\item[] $|P^\perp \cap \cO| = 0$, 
\item[] $|P^\perp \cap \Sigma_1| = |P^\perp \cap \Sigma_2| = \frac{q+1}{2}$, 
\item[] $|P^\perp \cap (\cQ^+(3, q^2) \setminus \cO)| = q^2+1$, 
\item[] $|P^\perp \cap \cH_1| = \frac{(q-1)(q^2-1)}{2} + q^2$, $|P^\perp \cap (\cS_j \setminus \cO)| = \frac{(q-1)(q^2-1)}{2}$, $j \in \{1, \dots, q\} \setminus \left\{\frac{q+1}{2}\right\}$, 
\item[] $|P^\perp \cap \cH_2| = |P^\perp \cap (\cE_k \setminus \cO)| = \frac{(q+1)(q^2+1)}{2}$, $k \in \{1, \dots, q-2\} \setminus \left\{\frac{q-1}{2}\right\}$, 
\item[] $|P^\perp \cap (\cE_0 \setminus \Sigma)| = |P^\perp \cap (\cE_{q-1} \setminus \Sigma)| = \frac{q^3-q}{2}$. 
\end{itemize} 
\end{lemma}
\begin{proof}
If $P \in \cH_1$, then the same arguments used in the previous lemma give $|P^\perp \cap \cO| = 0$, $|P^\perp \cap \cQ^+(3, q^2)| = q^2+1$, $|P^\perp \cap \Sigma_1| = |P^\perp \cap \Sigma_2| = (q+1)/2$, $|P^\perp \cap (\cE_0 \setminus \Sigma)| = |P^\perp \cap (\cE_{q-1} \setminus \Sigma)| = (q^3-q)/2$, $|P^\perp \cap \cH_2| = |P^\perp \cap (\cE_k \setminus \cO)| = (q+1)(q^2+1)/2$. In this case $|P^\perp \cap \cH(3, q^2)| = q^3+q^2+1$ and hence $|P^\perp \cap \cH_1| = (q-1)(q^2-1)/2 + q^2$. We claim that $|P^\perp \cap (\cS_j \setminus \cO)| = (q-1)(q^2-1)/2$. Indeed, if $P_1$ belongs to $\cS_j \setminus \cO$, then $|P_1^\perp \cap \cH(3, q^2)| = q^3+1$, since $P_1 \notin \cH(3, q^2)$ and $P_1^\perp \cap \cH_2 = (q+1)(q^2+1)/2$, by the previous lemma. Therefore $|P_1^\perp \cap \cH_1| = (q-1)(q^2-1)/2$, which in turn implies that through the point $P_1 \in (\cS_j \setminus \cO)$ there pass $(q-1)(q^2-1)/2$ planes $P_2^\perp$, where $P_2 \in \cH_1$. Double counting the pairs $(P_1, P_2^\perp)$, where $P_1 \in \cS_j \setminus \cO$, $P_2 \in \cH_1$, $P_1 \in P_2^\perp$, yields $|\cH_1| \times |P_2^\perp \cap (\cS_j \setminus \cO)| = |\cS_j \setminus \cO| \times (q-1)(q^2-1)/2$, as required.
\end{proof}

\begin{lemma}
If $P \in \cE_0 \setminus \Sigma$, then
\begin{itemize}
\item[] $|P^\perp \cap \cO| = 1$, 
\item[] $|P^\perp \cap \Sigma_1| = q$, $|P^\perp \cap \Sigma_2| = 0$, if $q \equiv -1 \pmod{4}$,
\item[] $|P^\perp \cap \Sigma_1| = 0$, $|P^\perp \cap \Sigma_2| = q$, if $q \equiv 1 \pmod{4}$,
\item[] $|P^\perp \cap (\cQ^+(3, q^2) \setminus \cO)| = q^2+1$, 
\item[] $|P^\perp \cap \cH_1| = |P^\perp \cap (\cS_j \setminus \cO)| = \frac{q^3-q^2}{2}$, $j \in \{1, \dots, q\} \setminus \left\{\frac{q+1}{2}\right\}$, 
\item[] $|P^\perp \cap \cH_2| = |P^\perp \cap (\cE_k \setminus \cO)| = \frac{q^3+q^2}{2}$, $k \in \{1, \dots, q-2\} \setminus \left\{\frac{q-1}{2}\right\}$, 
\item[] $|P^\perp \cap (\cE_0 \setminus \Sigma)| = \frac{q^3+q^2}{2}-q$, $|P^\perp \cap (\cE_{q-1} \setminus \Sigma)| = \frac{q^3+q^2}{2}$, if $q \equiv -1 \pmod{4}$, 
\item[] $|P^\perp \cap (\cE_0 \setminus \Sigma)| = \frac{q^3+q^2}{2}$, $|P^\perp \cap (\cE_{q-1} \setminus \Sigma)| = \frac{q^3+q^2}{2}-q$, if $q \equiv 1 \pmod{4}$. 
\end{itemize} 
If $P \in \cE_{q-1} \setminus \Sigma$, then
\begin{itemize}
\item[] $|P^\perp \cap \cO| = 1$, 
\item[] $|P^\perp \cap \Sigma_1| = 0$, $|P^\perp \cap \Sigma_2| = q$, if $q \equiv -1 \pmod{4}$,
\item[] $|P^\perp \cap \Sigma_1| = q$, $|P^\perp \cap \Sigma_2| = 0$, if $q \equiv 1 \pmod{4}$,
\item[] $|P^\perp \cap (\cQ^+(3, q^2) \setminus \cO)| = q^2+1$, 
\item[] $|P^\perp \cap \cH_1| = |P^\perp \cap (\cS_j \setminus \cO)| = \frac{q^3-q^2}{2}$, $j \in \{1, \dots, q\} \setminus \left\{\frac{q+1}{2}\right\}$, 
\item[] $|P^\perp \cap \cH_2| = |P^\perp \cap (\cE_k \setminus \cO)| = \frac{q^3+q^2}{2}$, $k \in \{1, \dots, q-2\} \setminus \left\{\frac{q-1}{2}\right\}$, 
\item[] $|P^\perp \cap (\cE_0 \setminus \Sigma)| = \frac{q^3+q^2}{2}$, $|P^\perp \cap (\cE_{q-1} \setminus \Sigma)| = \frac{q^3+q^2}{2}-q$, if $q \equiv -1 \pmod{4}$, 
\item[] $|P^\perp \cap (\cE_0 \setminus \Sigma)| = \frac{q^3+q^2}{2}-q$, $|P^\perp \cap (\cE_{q-1} \setminus \Sigma)| = \frac{q^3+q^2}{2}$, if $q \equiv 1 \pmod{4}$. 
\end{itemize} 
\end{lemma}
\begin{proof}
If $P \in \cE_0 \setminus \Sigma$, then $P^\perp \cap \Sigma = r$ is a Baer subline $|r \cap \cO| = 1$ and $|P^\perp \cap \cQ^+(3, q^2)| = q^2+1$. Let $\tilde{r}$ be the unique line of $\PG(3, q^2)$ containing $r$ and assume $q \equiv -1 \pmod{4}$. By Lemma~\ref{lemma:known} $\tilde{r} \in \cT_1$, $|P^\perp \cap \Sigma_1| = q$ and there are exactly $(q^3-q^2)/2$ or $(q^3+q^2)/2$ lines of $\cT_1$ or $\cT_2$ meeting $P^\perp$ in a point not of $\tilde{r}$. Hence $|P^\perp \cap (\cE_0 \setminus \Sigma)| = (q^3+q^2)/2-q$ and $|P^\perp \cap (\cE_{q-1} \setminus \Sigma)| = (q^3+q^2)/2$. Similarly by Lemma~\ref{property_k1} and Lemma~\ref{property_k} there are $(q^3+q^2)/2$ lines of $\cL$ and of $\cL_k$ intersecting $P^\perp \setminus \tilde{r}$ in one point. It follows that $|P^\perp \cap \cH_2| = |P^\perp \cap (\cE_k \setminus \cO)| = (q^3+q^2)/2$ and $|P^\perp \cap \cH_1| = (q^3-q^2)/2$. Let $V = r \cap \cO$. The $q^2$ lines of $P^\perp$ through $V$ distinct from $\tilde{r}$ are secants to $\cQ^+(3, q^2)$ and meets $\cO$ precisely in $V$. Hence $|P^\perp \cap (\cS_j \setminus \cO)| = (q^3-q^2)/2$ by Lemma~\ref{property_j}. The proof in the remaining cases is analogous.
\end{proof}

The results proved in the lemmas of this subsection are summarized in Table~\ref{table1}. For a point-orbit $\cA$ we denote by $\cA^\perp$ the orbit on planes given by $\cA^\perp = \{P^\perp \mid P \in \cP\}$. 

\begin{table}[t]\caption{Point-orbit distributions of the planes under the action of $K$.}\label{table1}
\vspace{0.2cm}
\resizebox{\columnwidth}{!}{%
\begin{tabular}{c||c|c|c|c|c|c|c|c|c|c}
   & $\cO$ & $\Sigma_1$ & $\Sigma_2$ & $\cQ^+(3,q^2)\setminus\cO$ & $\cH_1$ & $\cH_2$ & $\cS_{j'} \setminus \cO$ & $\cE_{k'} \setminus \cO$ & $\cE_0 \setminus \Sigma$ & $\cE_{q-1} \setminus \Sigma$ \\ 
   \hline
   \hline
$\cO^\perp$ & $1$ & $\frac{q^2+q}{2}$ & $\frac{q^2+q}{2}$ & $2q^2$ & $0$ & $q^3+q^2$ & $0$ & $q^3+q^2$ & $\frac{q^3-q}{2}$ & $\frac{q^3-q}{2}$ \\ 
\hline
$\Sigma_1^\perp$ & $q+1$ & \makecell{$q \equiv 1 \pmod{4}$: \\ $\frac{q^2-q}{2}$ \\ $q \equiv -1 \pmod{4}$: \\ $\frac{q^2+q}{2}$} & \makecell{$q \equiv 1 \pmod{4}$: \\ $\frac{q^2+q}{2}$ \\ $q \equiv -1 \pmod{4}$: \\ $\frac{q^2-q}{2}$} & $q^2-q$ & $\frac{q^3-q}{2}$ & $\frac{q^3-q}{2}$ & $\frac{q^3-q}{2}$ & $\frac{q^3-q}{2}$ & \makecell{$q \equiv 1 \pmod{4}$: \\ $0$ \\ $q \equiv -1 \pmod{4}$: \\ $q^3-q$} & \makecell{$q \equiv 1 \pmod{4}$: \\ $q^3-q$ \\ $q \equiv -1 \pmod{4}$: \\ $0$}  \\            
\hline
$\Sigma_2^\perp$ & $q+1$ & \makecell{$q \equiv 1 \pmod{4}$: \\ $\frac{q^2+q}{2}$ \\ $q \equiv -1 \pmod{4}$: \\ $\frac{q^2-q}{2}$} & \makecell{$q \equiv 1 \pmod{4}$: \\ $\frac{q^2-q}{2}$ \\ $q \equiv -1 \pmod{4}$: \\ $\frac{q^2+q}{2}$} & $q^2-q$  & $\frac{q^3-q}{2}$ & $\frac{q^3-q}{2}$ & $\frac{q^3-q}{2}$ & $\frac{q^3-q}{2}$ & \makecell{$q \equiv 1 \pmod{4}$: \\ $q^3-q$ \\ $q \equiv -1 \pmod{4}$: \\ $0$} & \makecell{$q \equiv 1 \pmod{4}$: \\ $0$ \\ $q \equiv -1 \pmod{4}$: \\ $q^3-q$} \\  
\hline
$(\cQ^+(3,q^2)\setminus\cO)^\perp$ & $2$ & $\frac{q-1}{2}$ & $\frac{q-1}{2}$ & $2q^2-1$ & $\frac{(q-1)(q^2+1)}{2}$ & $\frac{(q+1)(q^2-1)}{2}$ & $\frac{(q-1)(q^2+1)}{2}$ & $\frac{(q+1)(q^2-1)}{2}$ & $\frac{q^3-q}{2}$ & $\frac{q^3-q}{2}$ \\ 
\hline
$\cH_1^\perp$ & $0$ & $\frac{q+1}{2}$ & $\frac{q+1}{2}$ & $q^2+1$ & $\frac{(q-1)(q^2-1)}{2}+q^2$ & $\frac{(q+1)(q^2+1)}{2}$ & $\frac{(q-1)(q^2-1)}{2}$ & $\frac{(q+1)(q^2+1)}{2}$ & $\frac{q^3-q}{2}$ & $\frac{q^3-q}{2}$ \\ 
\hline
$\cH_2^\perp$ & $2$ & $\frac{q-1}{2}$ & $\frac{q-1}{2}$ & $q^2-1$  & $\frac{(q-1)(q^2+1)}{2}$ & $\frac{(q+1)(q^2-1)}{2}+q^2$ & $\frac{(q-1)(q^2+1)}{2}$ & $\frac{(q+1)(q^2-1)}{2}$ & $\frac{q^3-q}{2}$ & $\frac{q^3-q}{2}$ \\ 
\hline
$(\cS_{j}\setminus \cO)^\perp$ & $0$ & $\frac{q+1}{2}$ & $\frac{q+1}{2}$ & $q^2+1$ & $\frac{(q-1)(q^2-1)}{2}$ & $\frac{(q+1)(q^2+1)}{2}$ & \makecell{$q \equiv 1 \pmod{4}$: \\ $\frac{(q-1)(q^2-1)}{2}+q^2$, $j' =  q+1-j$ \\ $\frac{(q-1)(q^2-1)}{2}$, \hspace{0.7cm} $j' \ne  q+1-j$ \\ $q \equiv -1 \pmod{4}$: \\ $\frac{(q-1)(q^2-1)}{2}+q^2$, $j' =  j$ \\ $\frac{(q-1)(q^2-1)}{2}$, \hspace{0.7cm} $j' \ne  j$} & $\frac{(q+1)(q^2+1)}{2}$ & $\frac{q^3-q}{2}$ & $\frac{q^3-q}{2}$ \\ 
\hline
$(\cE_k \setminus \cO)^\perp$ & $2$ & $\frac{q-1}{2}$ & $\frac{q-1}{2}$ & $q^2-1$ & $\frac{(q-1)(q^2+1)}{2}$ & $\frac{(q+1)(q^2-1)}{2}$ & $\frac{(q-1)(q^2+1)}{2}$ & \makecell{$\frac{(q+1)(q^2-1)}{2} + q^2$, $k' = q-1-k$ \\ $\frac{(q+1)(q^2-1)}{2}$, \hspace{0.7cm} $k' \ne q-1-k$} & $\frac{q^3-q}{2}$ & $\frac{q^3-q}{2}$ \\
\hline
$(\cE_0 \setminus \Sigma)^\perp$ & $1$ & \makecell{$q \equiv 1\pmod{4}$: \\ $0$ \\ $q \equiv -1\pmod{4}:$ \\ $q$} & \makecell{$q \equiv 1\pmod{4}$: \\ $q$ \\ $q \equiv -1\pmod{4}:$ \\ $0$} & $q^2$ & $\frac{q^3-q^2}{2}$ & $\frac{q^3+q^2}{2}$ & $\frac{q^3-q}{2}$ & $\frac{q^3+q}{2}$ & \makecell{$q \equiv 1\pmod{4}$: \\ $\frac{q^3+q^2}{2}$ \\ $q \equiv -1\pmod{4}:$ \\ $\frac{q^3+q^2}{2}-q$} & \makecell{$q \equiv 1 \pmod{4}$: \\ $\frac{q^3+q^2}{2}-q$ \\ $q \equiv -1\pmod{4}:$ \\ $\frac{q^3+q^2}{2}$} \\ 
\hline
$(\cE_{q-1} \setminus \Sigma)^\perp$ & $1$ & \makecell{$q \equiv 1\pmod{4}$: \\ $q$ \\ $q \equiv -1\pmod{4}:$ \\ $0$} & \makecell{$q \equiv 1\pmod{4}$: \\ $0$ \\ $q \equiv -1\pmod{4}:$ \\ $q$} & $q^2$ & $\frac{q^3-q^2}{2}$ & $\frac{q^3+q^2}{2}$ & $\frac{q^3-q}{2}$ & $\frac{q^3+q}{2}$ & \makecell{$q \equiv 1\pmod{4}$: \\ $\frac{q^3+q^2}{2}-q$ \\ $q \equiv -1\pmod{4}:$ \\ $\frac{q^3+q^2}{2}$} & \makecell{$q \equiv 1 \pmod{4}$: \\ $\frac{q^3+q^2}{2}$ \\ $q \equiv -1\pmod{4}:$ \\ $\frac{q^3+q^2}{2}-q$} \\ 
\end{tabular}%
}
\end{table}

\begin{theorem}\label{quasiHermitianTHM}
If $j \in \{1, \dots, q\} \setminus \{(q+1)/2\}$ and $k \in \{1, \dots, q-2\} \setminus \{(q-1)/2\}$, then the sets $\cS_j \cup \cE_k$, $\cH_1 \cup \cE_k$ and $\cS_j \cup \cH_2$ are quasi-Hermitian surfaces.
\end{theorem}
\begin{proof}
The result follows by amalgamating three columns of Table~\ref{table1}.
\end{proof}

\section{The known examples and the isomorphism issue}\label{SectionNonIso}
To the best of our knowledge there are three known constructions of quasi-Hermitian surfaces in $\PG(3, q^2)$ which we describe below.

\paragraph*{First construction}

Let $\pi$ be a plane tangent to $\cH(3, q^2)$ at the point $P$. Let $g_i$, $1 \le i \le q+1$, be the $q+1$ generators of $\cH(3, q^2)$ through $P$ and consider $q+1$ lines $\ell_1, \dots, \ell_{q+1}$ of $\pi$ through $P$. The set 
\begin{align*}
& \cV_1 = \bigcup_i \ell_i \cup \left( \cH(3, q^2) \setminus \bigcup_{i} g_i \right)
\end{align*}
is a quasi-Hermitian surface of $\PG(3, q^2)$, see \cite[Theorem 3]{DS}.  Let $z$ be the number of members in common between $\{g_1, \dots, g_{q+1}\}$ and $\{\ell_1, \dots, \ell_{q+1}\}$. Of course $\cV_1 = \cH(3, q^2)$ if and only if $z = q+1$. %Assume that $0 \le z < q+1$. 
The following result can be easily deduced.

\begin{lemma}\label{v1}
There are precisely $zq^3+q+1$ lines of $\PG(3, q^2)$ contained in $\cV_1$. There are $q^5$ points of $\cV_1$ incident with exactly $z$ lines of $\cV_1$, $zq^2+1$ points of $\cV_1$ contained in $q+1$ lines of $\cV_1$ and $(q+1-z)q^2$ points of $\cV_1$ lying on one line of $\cV_1$.
\end{lemma}

\paragraph*{Second construction}

Fix $\alpha \in \F_{q^2} \setminus \{0\}$, $\beta \in \F_{q^2} \setminus \F_q$, with $4 \alpha^{q+1} + (\beta^q - \beta)^2 \ne 0$. The set
\begin{align*}
& \cV_2 = \{(1, x, y, z) \mid x,y,z \in \F_{q^2}, G(x,y,z) = 0\} \cup \{(0, x, y, x) \mid x,y,z \in \F_{q^2}, x^{q+1} + y^{q+1} = 0\}, \\
& \mbox{ where } G(x,y,z) = z^q - z + \alpha^q (x^{2q} + y^{2q}) - \alpha (x^2+y^2) - (\beta^q - \beta) (x^{q+1}+y^{q+1}),
\end{align*}
is a quasi-Hermitian surface \cite{ACK}. In a similar way, further constructions were exhibited in \cite{A, ACK} in the case of even characteristic.  

\begin{lemma}[\cite{AG}]\label{v2}
The following hold.
\begin{itemize}
\item If $q \equiv 1 \pmod{4}$, there are precisely $2q^2+q+1$ lines of $\PG(3, q^2)$ contained in $\cV_2$. There are $q^5$ points of $\cV_2$ incident with exactly $2$ lines of $\cV_2$, $2q^2+1$ points of $\cV_2$ contained in $q+1$ lines of $\cV_2$ and $q^3-q^2$ points of $\cV_2$ lying on one line of $\cV_2$. 
\item If $q \equiv -1 \pmod{4}$, there are precisely $q+1$ lines of $\PG(3, q^2)$ contained in $\cV_2$. There are $q^5$ points of $\cV_2$ lying on no line of $\cV_2$, $q^3+q^2$ points of $\cV_2$ contained in exactly one line of $\cV_2$ and one point of $\cV_2$ incident with $q+1$ line of $\cV_2$. 
\end{itemize}
\end{lemma}

\paragraph*{Third construction}

Let $\Sigma$ be a Baer subgeometry of $\PG(3, q^2)$ and let $\cQ$ be a non-degenerate quadric of $\Sigma$. Let $\cL$ be the set of lines of $\PG(3, q^2)$ having $q+1$ points in common with $\Sigma$ and intersecting $\cQ$ in either one or $q+1$ points. Then 
\begin{align}
& \cV_3 = \bigcup_{\ell \in \cL} \ell \label{h3}
\end{align}  
is a quasi-Hermitian surface \cite{P}. More generally, if $\cL$ is a set of lines of $\PG(3, q^2)$, having $q+1$ points in common with $\Sigma$ and such that through each point of $\Sigma$ there pass $q+1$ lines of $\cL$, then $\cV_3$ defined as in \eqref{h3} is a quasi-Hermitian surface, see \cite{CPa}. 

\begin{lemma}\label{v3}
There are at least $(q+1)(q^2+1)$ lines of $\PG(3, q^2)$ contained in $\cV_3$. Through one of the $q^5-q$ points of $\cV_3 \setminus \Sigma$ there is at least one line of $\cV_3$, whereas each of the $(q+1)(q^2+1)$ points of $\cV_3 \cap \Sigma$ lies on at least $q+1$ lines of $\cV_3$.
\end{lemma}

Denote by $\cV_4$ a set of the form $\cS_j \cup \cE_k$ or $\cH_1 \cup \cE_k$ or $\cS_j \cup \cH_2$. The following result is a consequence of Lemma~\ref{property_k1}, Lemma~\ref{property_k} and Lemma~\ref{property_k2} if $q >3$ or of direct computations if $q = 3$.
\begin{prop}
There are precisely $(q+1)(q^2+1)$ lines of $\PG(3, q^2)$ contained in $\cV_4$. There is no line of $\cV_4$ through a point of $\cS_j \setminus \cO$ or $\cH_2$, there are exactly two lines of $\cV_4$ through a point of $\cE_k \setminus \cO$ or $\cH_2$ and exactly $q+1$ lines of $\cV_4$ through a point of $\cO$. 
\end{prop}

\begin{cor}
$\cV_4$ is not projectively equivalent to $\cV_i$, $i = 1,2,3$.
\end{cor}
%\begin{cor}
%The quasi-Hermitian varieties of the type $\cV_4$ are pairwise not projectively equivalent.
%\end{cor}

\section{Further combinatorial invariants}\label{sectionpointorbitdistribution}

In this section we look at $\cO$ as the set of $\mathbb{F}_{q^2}$-rational points of the rational curve 
\begin{align*}
\cC=\{(1,t,t^q,t^{q+1}) \mid t \in \bar{\mathbb{F}}_{q^2}\} \cup \{(0,0,0,1)\}
\end{align*} 
and we investigate more closely the orbit structure of the setwise stabilizer of the curve $\cC$ over $\mathbb{F}_{q^2}$. Explicit equations for $\cC$ are given by
\begin{equation*}
\begin{cases}
X_2^{q+1}=X_1^qX_4,\\
X_3^{q+1}=X_1X_4^q,\\
X_1X_4=X_2X_3.
\end{cases}
\end{equation*}
A further curve, $\cC'$, given by 
\begin{equation*}
\begin{cases}
X_3^{q+1}=X_1^qX_4,\\
X_2^{q+1}=X_1X_4^q,\\
X_1X_4=X_2X_3,
\end{cases}
\end{equation*}
shares with $\cC$ the set of $\mathbb{F}_{q^2}$-rational points, namely $\cC'(\mathbb{F}_{q^2})=\cC(\mathbb{F}_{q^2})=\cO$. By tangent plane at a point $P\in\cO$ ($\cO$-tangent plane for short), we mean the tangent plane of the curve $\cC$ at $P$, and by $\cC$-tangent ($\cC'$-tangent) line at a point $P\in\cO$ we mean the tangent line of the curve $\cC$ ($\cC'$) at $P$. The involution $\iota: (X_1, X_2, X_3, X_4) \longmapsto (X_1, X_3, X_2, X_4)$ fixes $\cO$ and maps $\cC$ to $\cC'$. %The quadric $\cQ^+(3,q^2)$ splits the Hermitian surface $\cH$ precisely in $\cC$ and $\tilde{\cC}$.
%The automorphism group of the rational curve $\cC$ is $\PGL(2,\bar{\mathbb{F}}_{q^2})$, and 
The $\mathbb{F}_{q^2}$-rational automorphism group of $\cC$ is $G \simeq \PGL(2, q^2) \le \PGL(4, q^2)$, whereas the set of all $\fqq$-rational projectivities fixing $\cO$ is the group $G' \simeq \PGL(2,q^2)\rtimes \langle \iota \rangle \le \PGL(4, q^2)$. In this section we take into account the action of $G$ or $G'$, rather than $K$. In particular, in Lemma~\ref{PSLtoPGL} we will show that $G$ actually merges distinct $K$-orbits on points of $\PG(3,q^2)$, whereas $G$ and $G'$ have the same point-orbits. In the first part of the section, we will compute the point-orbit distributions of the planes under the action of $G$. Then we deal with the $G$-orbits on lines, which turns out to be a much more difficult problem to tackle.

For ease of notation, we define 
\begin{align*}
& \tilde{\cS}_j = (\cS_j \cup \cS_{q+1-j}) \setminus \cO, \quad j \in \{1, \dots, (q-1)/2\}, \\
& \tilde{\cE}_k = (\cE_k\cup\cE_{q-1-k}) \setminus \cO, \quad k \in \{1, \dots,(q-3)/2\}, \\ 
& \tilde{\cE}_0 = (\cE_0 \cup\cE_{q-1}) \setminus \Sigma.
\end{align*}
\begin{prop}\label{PSLtoPGL}
The groups $G$ and $G'$ have $q+4$ orbits on points of $\PG(3, q^2)$:
\begin{itemize}
\item[i)] $\cO$;
\item[ii)] $\Sigma\setminus \cO$;
\item[iii)] $\cQ^+(3, q^2) \setminus \cO$;
\item[iv)] the two orbits $\cH_1, \cH_2$ on points of $\cH(3, q^2) \setminus \cO$;
\item[v)] $\tilde{\cS}_j$, $j \in \{1, \dots, (q-1)/2\}$;
\item[vi)] $\tilde{\cE}_k$, $k \in \{1, \dots,(q-3)/2\}$;
\item[vii)] $\tilde{\cE}_0$;
\end{itemize}
\end{prop} 
\begin{proof}
It is not hard to see that there are maps in $G$ sending $S_1$ to $S_2$, $R_j$ to $R_{q+1-j}$, $Q_k$ to $Q_{q-1-k}$ and $T_1$ to $T_2$. %On the other hand, Lemma \ref{numberoftgplanes} implies that orbits $\cS_j\setminus \cO$ and $\cE_k\setminus\cO$ will not merge. 
Following the proof of Lemma \ref{lemma:stab}, by using the group $G$ instead of $K$, gives the result. The claim on the group $G'$ follows from the fact that the map $\iota$ preserves $\cO$, $\Sigma$, $\cQ^+(3,q^2)$, $\cH_1$, $\cH_2$ and fixes each of the following points: $ T_1, R_1,\dots, R_{(q-1)/2}, Q_0,\dots,Q_{(q-3)/2}$.
\end{proof}

The point-orbit distributions of the planes under the action of $G$ can be recovered taking into account Proposition~\ref{PSLtoPGL} and Table~\ref{table1}. The related results are summarized in Table~\ref{table2}. 

\begin{table}[t]\caption{Point-orbit distributions of the planes under the action of $G$.}\label{table2}
\vspace{0.2cm}
\resizebox{\columnwidth}{!}{%
\begin{tabular}{c||c|c|c|c|c|c|c|c}
   & $\cO$ & $\Sigma$ & $\cQ^+(3,q^2)\setminus\cO$ & $\cH_1$ & $\cH_2$ & $\tilde{\cS}_{j'}$ & $\tilde{\cE}_{k'}$ & $\tilde{\cE}_0$ \\ 
   \hline
   \hline
$\cO^\perp$ & $1$ & $q^2+q$ & $2q^2$ & $0$ & $q^3+q^2$ & $0$ & $2(q^3+q^2)$ & $q^3-q$ \\ 
\hline
$(\Sigma \setminus \cO)^\perp$ & $q+1$ & $q^2$ & $q^2-q$ & $\frac{q^3-q}{2}$ & $\frac{q^3-q}{2}$ & $q^3-q$ & $q^3-q$ & $q^3-q$ \\ 
\hline
$(\cQ^+(3,q^2)\setminus\cO)^\perp$ & $2$ & $q-1$ & $2q^2-1$ & $\frac{(q-1)(q^2+1)}{2}$ & $\frac{(q+1)(q^2-1)}{2}$ & $(q-1)(q^2+1)$ & $(q+1)(q^2-1)$ & $q^3-q$ \\ 
\hline
$\cH_1^\perp$ & $0$ & $q+1$ & $q^2+1$ & $\frac{(q-1)(q^2-1)}{2}+q^2$ & $\frac{(q+1)(q^2+1)}{2}$ & $(q-1)(q^2-1)$ & $(q+1)(q^2+1)$ & $q^3-q$ \\ 
\hline
$\cH_2^\perp$ & $2$ & $q-1$ & $q^2-1$ & $\frac{(q-1)(q^2+1)}{2}$ & $\frac{(q+1)(q^2-1)}{2}+q^2$ & $(q-1)(q^2+1)$ & $(q+1)(q^2-1)$ & $q^3-q$ \\ 
\hline
$\tilde{\cS}_{j}^\perp$ & $0$ & $q+1$ & $q^2+1$ & $\frac{(q-1)(q^2-1)}{2}$ & $\frac{(q+1)(q^2+1)}{2}$ & \makecell{$(q-1)(q^2-1)+q^2$, $j' = j$ \\ $(q-1)(q^2-1)$, \hspace{0.7cm} $j' \ne j$} & $(q+1)(q^2+1)$ & $q^3-q$ \\ 
\hline
$\tilde{\cE}_k^\perp$ & $2$ & $q-1$ & $q^2-1$ & $\frac{(q-1)(q^2+1)}{2}$ & $\frac{(q+1)(q^2-1)}{2}$ & $(q-1)(q^2+1)$ & \makecell{$(q+1)(q^2-1) + q^2$, $k' = k$ \\ $(q+1)(q^2-1)$, \hspace{0.7cm} $k' \ne k$} & $q^3-q$ \\
\hline
$\tilde{\cE}_0^\perp$ & $1$ & $q$ & $q^2$ & $\frac{q^3-q^2}{2}$ & $\frac{q^3+q^2}{2}$ & $q^3-q$ & $q^3+q$ & $q^3+q^2-q$ \\ 
\end{tabular}%
}
\end{table}

\subsection{On the line \texorpdfstring{$G$}{}-orbits}\label{line-orbits}
In this subsection we investigate the orbits of $G$ on the lines of $\PG(3,q^2)$. Some of the line orbits are easy to spot. Indeed, the fact that $G$ is triply transitive on points of $\cC(\fqq)$ implies that the $\cC$-tangents and the $\cC'$-tangents form two orbits of size $q^2+1$. Note that that these $2(q^2+1)$ lines are those of $\cQ^+(3, q^2)$. By Lemma~\ref{lemma:known} the extended sublines of $\Sigma$ that are either secant or external to $\cO$ form two orbits of size $q^2(q^2+1)/2$ and the extended sublines of $\Sigma$ that are tangent to $\cO$ give rise to another orbit of size $(q+1)(q^2+1)$. As a consequence of Lemma~\ref{property_k} or Lemma~\ref{property_k1}, there are $q-1$ further orbits of size $(q+1)(q^2+1)$, given by the lines contained in $\cE_k$, $k \in \{1,\dots, q-1\}\setminus \{(q-1)/2\}$ and in $\cH_2$. By \cite{CP} the group $G$ has two further orbits of size $q^2(q^2-1)/2$ on lines of $\cH(3, q^2)$ besides $\cL$. Moreover, by \cite{KNS} these two line orbits can be obtained as follows. The line joining any point of $\cC(\mathbb{F}_{q^4})\setminus \cC(\fqq)$ or of $\cC'(\mathbb{F}_{q^4})\setminus \cC(\fqq)$ with its image under the Frobenius automorphism of $\mathbb{F}_{q^4}$ fixing $\fqq$ is an $\fqq$-rational line. Since in both cases the number of such points is $q^2(q^2-1)$, we get as many as $q^2(q^2-1)/2$ lines. The fact that they actually are on the same orbit follows by the transitivity of $G$ on the set of points $\cC(\mathbb{F}_{q^4})\setminus\cC(\fqq)$ and $\cC'(\mathbb{F}_{q^4})\setminus\cC(\fqq)$. %, which is again a consequence of the triple transitivity of $G$. The dual of this orbit, given by imaginary axis, provide a further orbit of size $q^2(q^2-1)/2$.

By using the Klein correspondence we can identify $q^2$ more orbits, as described below.
% we have q^2+q-1+2+2+2=q^2+q+7 orbits so far
Let $\phi$ be the Klein correspondence mapping the set of lines in $\PG(3,q)$ to the set of points lying on the Klein quadric $\cK\subset \PG(5,q^2)$.
The map $\phi$ induces an action of $G$ on the points of $\PG(5,q^2)$, preserving $\cK$. In particular, the map induced by $A = \begin{pmatrix} 
a & b \\
c & d 
\end{pmatrix} \in \GL(2, q^2)$ acts on the points of $\PG(5,q^2)$ as the projectivity induced by the matrix
\begin{equation*}
\begin{pmatrix}
a^{2q} & 0 & a^qb^q & -a^qb^q & 0 & b^{2q} \\
0 & a^2\Delta^{q-1} & ab\Delta^{q-1} & ab\Delta^{q-1} & -b^2\Delta^{q-1} & 0 \\
a^qc^q & ac\Delta^{q-1} & \frac{a^{q+1}d^{q+1}-b^{q+1}c^{q+1}}{\Delta} &  \frac{a^qbcd^q-ab^qc^qd}{\Delta} & -bd\Delta^{q-1} & b^qd^q\\
-a^qc^q & ac\Delta^{q-1} & \frac{a^qbcd^q-ab^qc^qd}{\Delta} & \frac{a^{q+1}d^{q+1}-b^{q+1}c^{q+1}}{\Delta} & -bd\Delta^{q-1} & -b^qd^q\\
0 & -c^2\Delta^{q-1} & -cd\Delta^{q-1} & -cd\Delta^{q-1} & d^2\Delta^{q-1} & 0 \\
c^{2q} & 0 & c^qd^q & -c^qd^q & 0 & d^{2q} \end{pmatrix}
\end{equation*}
yielding a group $\bar{G} \simeq \PGL(2, q^2) \le \PGO^+(6, q^2)$. Here $\Delta=ad-bc\neq 0$.

\begin{proposition}
Let $R_\w \in \cK$ be the point $(0, \w, 1, 0, 0, 1)$. The group $\bar{G}$ has two point orbits of length $q^6-q^2$ with representatives $R_\w$, $\w \in \{0,1\}$, and $q^2-2$ point orbits of length $(q^6-q^2)/2$ with representatives $R_\w$, $\w \in \F_{q^2} \setminus \{0,1\}$. 
\end{proposition}
\begin{proof}
We are going to compute the size of the stabilizer $\bar{G}_{R_\w}$ of $R_\w$ in $\bar{G}$.
If the projectivity $g \in \bar{G}$ induced by $A = \begin{pmatrix} 
a & b \\
c & d 
\end{pmatrix}$ fixes $R_\w$ then, for a non-zero $\lambda\in \F_{q^2}$, the following hold true:

\begin{equation}\label{syPw}
\begin{cases}
a^qb^q+b^{2q}=0 \\
\w a^2\Delta^{q-1}+ab\Delta^{q-1}=\lambda \w \\
\w ac\Delta^{q-1}+\frac{a^{q+1}d^{q+1}-b^{q+1}c^{q+1}}{\Delta}+b^qd^q=\lambda \\
\w ac\Delta^{q-1}+\frac{a^qbcd^q-ab^qc^qd}{\Delta}-b^qd^q=0 \\
-\w c^2\Delta^{q-1}-cd\Delta^{q-1}=0 \\
c^qd^q+d^{2q}=\lambda
\end{cases}
\end{equation}
where $\Delta=ad-bc \neq 0$. Assume first $\w = 0$. Then $ab = cd = 0$, which implies $b = c = 0$, and $a = d$, with $a, d \neq 0$. Hence in this case $\bar{G}_{R_0} = \{id\}$. Let $\w \neq 0$. From the second equation and the sixth equation, $a\neq 0$ and $d\neq 0$ hold. If $b=0$ then, from the fourth equation, $c=0$ and then $d=a$, that is $g=id$.
If $c=0$, then $b=0$ and again $g=id$. %If $b=c=0$, a contradiction arises from the third and fourth equations.
Assume $b,c\neq 0$, and $a=1$. Then $b^q=-1$, that is $b=-1$ and $\Delta=c+d$. From the fifth and sixth equations it follows $d=-\w c$. If $\w = 1$, then $\Delta=0$, a contradiction. Hence $\bar{G}_{R_1} = \{id\}$. If $\w \neq 1$, the third equation yields $c = \w^{-1}$. In this case the unique non-trivial projectivity of $\bar{G}_{R_\w}$ is induced by $A = \begin{pmatrix} 
1 & -1 \\
\w^{-1} & -1 
\end{pmatrix}$.
 
Replacing $\w$ with $\w'$, $\w \ne \w'$, in the right hand side of the second equation and performing the same computations leads to a contradiction, showing that these $\bar{G}$-orbits are distinct. 
\end{proof}

\section{Concluding remarks}

In this paper we provided a construction of quasi-Hermitian surfaces by investigating the action of $\mathrm{P\Omega}^-(4, q) \simeq \PSL(2, q^2) \le \PGL(4, q^2)$ on points of $\PG(3, q^2)$, $q$ odd. We strongly believe that similar results can be achieved by considering the action of $\mathrm{P\Omega}^\pm(2n+2, q) \le \PGL(2n+2, q^2)$ on points of $\PG(2n+1, q^2)$. 

In subsection \ref{line-orbits}, we identified as many as $q^2+q+5$ line $G$-orbits on lines of $\PG(3, q^2)$, where $G \simeq \PGL(2, q^2) \le \PGL(4, q^2)$. Computational data suggest that $G$ has $2q^2+2q+4$ orbits on lines of $\PG(3,q^2)$.

\bigskip

{\footnotesize
\noindent\textit{Acknowledgments.}
This work was supported by the Italian National Group for Algebraic and Geometric Structures and their Applications (GNSAGA-- INdAM). The second author acknowledges the support by the Irish Research Council, grant n. GOIPD/2022/307. 
}

\bibliographystyle{plain}

\bibliography{quasiHermitianBiblio}

\end{document}